\documentclass[a4paper,11pt]{article}
\pdfoutput=1

\usepackage[a4paper,tmargin=3truecm,bmargin=3truecm,rmargin=2.5truecm,
lmargin=2.5truecm,twoside,verbose=true]{geometry}
\usepackage{cancel,graphicx}

\usepackage{amsmath,amssymb}
\usepackage[amsmath, hyperref, thmmarks]{ntheorem}
\usepackage[all]{xypic}
\usepackage[pdftex]{hyperref}
\usepackage[english]{babel}

\numberwithin{equation}{section}

\allowdisplaybreaks[1]

\newcommand\defin{\bf}

\newcommand\caD{{\mathcal D}}
\newcommand\caE{{\mathcal E}}
\newcommand\caF{{\mathcal F}}

\newcommand\caL{{\mathcal L}}
\newcommand\caP{{\mathcal P}}

\newcommand\caR{{\mathcal R}}

\newcommand\caM{{\mathcal M}}
\newcommand\caS{{\mathcal S}}
\newcommand\caO{{\mathcal O}}
\newcommand\caU{{\mathcal U}}

\newcommand\gone{{ \mathchoice {1\mskip-4mu\mathrm{l} } {1\mskip-4mu\mathrm{l} }{1\mskip-4.5mu\mathrm{l} } {1\mskip-5mu\mathrm{l}} }}
\newcommand\gR{{\mathbb R}}

\newcommand\gC{{\mathbb C}}

\newcommand\R{{\mathbb R}}

\newcommand\gS{{\mathbb S}}
\newcommand\gN{{\mathbb N}}
\newcommand\gZ{{\mathbb Z}}

\newcommand\bfm{{\mathsf{\mathbf m}}}

\newcommand\algzero{{\mathsf 0}}

\newcommand\algA{{\mathbf A}}

\newcommand\ehH{\mathcal H}

\newcommand\ka{{\mathfrak a}}
\renewcommand\k{{\mathfrak k}}
\newcommand\n{{\mathfrak n}}
\newcommand\z{{\mathfrak z}}
\renewcommand\a{{\mathfrak a}}
\renewcommand\b{{\mathfrak g}}
\newcommand\g{{\mathfrak G}}
\newcommand\kB{{\mathfrak B}}

\newcommand\kq{{\mathfrak q}}
\newcommand\kg{{\mathfrak g}}

\newcommand\kh{{\mathfrak h}}

\newcommand\ks{{\mathfrak s}}

\newcommand\ee{{\epsilon}}
\newcommand\ad{{\text{\textup{ad}}}}
\newcommand\Ad{{\text{\textup{Ad}}}}

\newcommand\fois{\mathord{\cdot}}
\DeclareMathOperator{\tr}{Tr} 

\DeclareMathOperator{\Aut}{\mathsf{Aut}}

\newcommand\Der{{\text{\textup{Der}}}}

\newcommand\dd{{\text{\textup{d}}}}

\newcommand\norm{\mathord{\parallel}}

\theoremsymbol{}
\theorembodyfont{\slshape}
\theoremheaderfont{\normalfont\bfseries}
\theoremseparator{}
\newtheorem{Theorem}{Theorem}[section]
\newtheorem{theorem}[Theorem]{Theorem}

\newtheorem{proposition}[Theorem]{Proposition}
\newtheorem{Lemma}[Theorem]{Lemma}
\newtheorem{lemma}[Theorem]{Lemma}

\newtheorem{corollary}[Theorem]{Corollary}

\theorembodyfont{\upshape}
\theoremsymbol{\ensuremath{\blacklozenge}}

\newtheorem{example}[Theorem]{Example}

\newtheorem{remark}[Theorem]{Remark}

\newtheorem{definition}[Theorem]{Definition}

\theoremstyle{nonumberplain}
\theoremheaderfont{\scshape}
\theorembodyfont{\normalfont}
\theoremsymbol{\ensuremath{\blacksquare}}

\newtheorem{proof}{Proof}
\qedsymbol{\ensuremath{_\blacksquare}}
\theoremclass{LaTeX}

\renewenvironment{thebibliography}[1]
         {\section*{References}\frenchspacing\small
          \begin{list}{[\arabic{enumi}]}
         {\usecounter{enumi}\parsep=2pt\topsep 0pt
         \settowidth{\labelwidth}{[#1]}
         \leftmargin=\labelwidth\advance\leftmargin\labelsep
         \rightmargin=0pt\itemsep=1pt\sloppy}}{\end{list}}

\title{Harmonic analysis on homogeneous complex bounded domains and noncommutative geometry\footnote{Work
supported by the Belgian Interuniversity Attraction Pole (IAP) within the framework ``Dynamics, Geometry and Statistical Physics'' (DYGEST).}}
\author{}
\author{Pierre Bieliavsky$^1$, Victor Gayral$^2$, Axel de Goursac$^1$, Florian Spinnler$^1$}

\begin{document}

\maketitle
\vspace*{-1cm}
\begin{center}
\textit{1: D\'epartement de Math\'ematiques, \\ Universit\'e Catholique de Louvain,\\ 
Chemin du Cyclotron, 2,\\
1348 Louvain-la-Neuve, Belgium\\
\vspace{2mm}
2: Universit\'e Reims Champagne-Ardenne\\
Laboratoire de Math\'ematiques\\
Moulin de la Housse - BP 1039\\
51687 Reims cedex 2, France\\
\vspace{2mm}
e-mails: \texttt{Pierre.Bieliavsky@uclouvain.be, victor.gayral@univ-reims.fr, Axelmg@melix.net, Florian.Spinnler@uclouvain.be}}
\end{center}%

\vskip 2cm

\begin{abstract}
We define and study a noncommutative Fourier transform on every homogeneous complex bounded domain. We then give an application in noncommutative differential geometry by defining noncommutative Baumslag-Solitar tori.
\end{abstract}


\vspace{2mm}


\vskip 1cm

\tableofcontents

\section{Introduction}

In \cite{BG}, the  authors developed a tracial symbolic pseudo-differential calculus on every Lie group $G$ whose Lie algebra $\kg$ is a normal $j$-algebras in the sense of Pyatetskii-Shapiro \cite{PS}. The class of such Lie groups is in one to one correspondence 
with the class of homogeneous complex bounded domains. Each of them carries a
left-invariant Kahler structure.

\noindent As a by-product, they obtained a $G$-equivariant continuous linear mapping between the Schwartz space $\caS(G)$ of such a Lie group and a sub-algebra of Hilbert-Schmidt operators
on a Hilbert irreducible unitary $G$-module. This yields a one-parameter family of noncommutative associative multiplications 
$\{\star_\theta\}_{\theta\in\gR}$ on the Schwartz space, each of them endowing 
$\caS(G)$ with a Fr\'echet nuclear algebra structure. Moreover the resulting family of Fr\'echet algebras $\{(\caS(G),\star_\theta)\}_{\theta\in\gR}$ deforms the commutative 
Fr\'echet algebra structure on $\caS(G)$ given by the pointwise multiplication of functions corresponding to the value $\theta=0$ of the deformation parameter. Note that such a program was achieved in \cite{Rieffel:1993} for abelian Lie groups and in \cite{Bieliavsky:2010su} for abelian Lie supergroups.
\medskip

In the present article, we construct a bijective intertwiner between every
noncommutative Fr\'echet  algebra $(\caS(G),\star_\theta)$ ($\theta\neq0$) and 
a convolution function algebra on the group $G$. The intertwiner's kernel consists 
in a complex valued smooth function $\caE$ on the group $G\times G$ that we 
call ``$\star$-exponential" because of its similar nature with objects defined in 
\cite{F} and studied in \cite{AC} in the context of the Weyl-Moyal quantization of co-adjoint orbits of exponential Lie groups.

\noindent We then prove that the associated smooth map
$$
\caE: G\to C^\infty(G)
$$
consists in a group-morphism valued in the multiplier 
(nuclear Fr\'echet) algebra $\caM_{\star_\theta}(G)$
of $(\caS(G),\star_\theta)$. The above group-morphism integrates the classical
moment mapping 
$$
\lambda:\kg\to\caM_{\star_\theta}(G)\;\cap\;C^\infty(G)
$$
associated with the (symplectic) action of $G$ on itself by left-translations.

\noindent Next, we modify the 2-point kernel $\caE$ by a power of the modular function of $G$ in such a way that the corresponding Fourier-type transform
consists in  a unitary operator $\caF$ on the Hilbert space of square integrable 
functions with respect to a left-invariant Haar measure on $G$.

\noindent As an application, we define a class of noncommutative tori associated
to generalized Bauslag-Solitar groups in every dimension.

\vspace{3mm}

\noindent {\bf Acknowledgment.} One of us, Pierre Bieliavsky, spent the academic year 1995-1996 at UC Berkeley as a post-doc in the group of Professor Joseph A. Wolf.
It is a great pleasure for P. B. to warmly thank Professor Wolf for his support not only when a young post-doc but constantly during P. B. 's career. The research presented in this note is closely related to the talk P. B. gave at the occasion of the West Coast Lie Theory Seminar in November 1995 when studying some early stage features of the non-formal $\star$-exponential \cite{BUCB}.

\section{Homogeneous complex bounded domains and $j$-algebras}

The theory of $j$-algebras was much developed by Pyatetskii-Shapiro \cite{PS} for studying in a Lie-algebraic way the structure and classification of bounded homogeneous --- not necessarily symmetric --- domains in $\gC^n$. A $j$-algebra is roughly the Lie algebra $\kg$ of a transitive Lie group of analytic automorphisms of the domain, together with the data of the Lie algebra $\k$ of the stabilizer of a point in the latter Lie group, an endomorphism $j$ of $\kg$ coming from the complex structure on the domain, and a linear form on $\kg$ whose Chevalley coboundary gives the $j$-invariant symplectic structure coming from the K\"ahler structure on the domain. Pyatetskii-Shapiro realized that, among the $j$-algebras corresponding to a fixed bounded homogeneous domain, there always is at least one whose associated Lie group acts simply transitively on the domain, and which is realizable as upper triangular real matrices. Thoses $j$-algebras have the structure of \emph{normal $j$-algebras} which we proceed to describe now.

\begin{definition}
A {\defin normal $j$-algebra} is a triple $(\kg,\alpha,j)$ where
\begin{enumerate} 
\item $\kg$ is a solvable Lie algebra which is split over the reals, i.e. $\ad_X$ has only real eigenvalues for all $X\in \kg$, 
\item $j$ is an endomorphism of $\kg$ such that $j^2=-Id_{\kg}$ and $[X,Y]+j[jX,Y]+j[X,jY]-[jX,jY]=0$, $\forall X,Y \in \kg$,

\item  $\alpha$ is a linear form on $\kg$ such that:
$\alpha([jX,X])>0$ if $X\neq 0$ and $\alpha([jX,jY])=\alpha([X,Y])$, $\forall X,Y \in\kg$.
\end{enumerate}
\end{definition}

\noindent If $\b'$ is a subalgebra of $\b$ which is invariant by $j$, then $(\b',\alpha|_{\b'},j|_{\b'})$ is again a normal $j$-algebra, said to be a {\defin $j$-subalgebra} of $(\b,\alpha,j)$. A $j$-subalgebra whose algebra is at the same time an ideal is called a {\defin $j$-ideal}.

\begin{remark}
To each simple Lie algebra $\g$ of Hermitian type (i.e. such that the center of the maximal compact algebra $\k$ has real dimension one) we can attach a normal $j$-algebra $(\b,\alpha,j)$ where 
\begin{enumerate}
\item
$\b$ is the solvable Lie algebra underlying the Iwasawa factor $\b=\a \oplus\n$
of an Iwasawa decomposition $\k \oplus\a \oplus\n$ of $\g$.
\item Denoting by $\mathbb{G}/K$ the Hermitean symmetric space associated to 
the pair $(\g,\k)$ and by $\mathbb{G}=KAN$ the Iwasawa group decomposition corresponding to $\k\oplus\a\oplus\n$, 
the global diffeomorphism:
$$
G:=AN\longrightarrow \mathbb{G}/K:g\mapsto gK\;,
$$
endows the group $G$ with an exact left-invariant symplectic structure 
as well as a compatible complex structure. The evaluations at the unit
element $e\in G$ of these tensor fields define the elements ${\bf \Omega}=\dd\alpha$ and 
$j$ at the Lie algebra level.
\end{enumerate}
\end{remark}

\noindent It is important to note that not every normal $j$-algebra arises this way. Indeed, it is with the help of the theory of $j$-algebras that Pyatetskii-Shapiro discovered the first examples of non-symmetric bounded homogeneous domains. Nevertheless, they can all be built from these ``Hermitian'' normal $j$-algebras by a semi-direct product process, as we recall now.

\begin{definition}
\label{def-elemjalg}
A normal $j$-algebra associated with a rank one Hermitean symmetric space (i.e.\ $\dim \a=1$) is called {\defin elementary}.
\end{definition}

\begin{Lemma}
Let $(V,\omega_0)$ be a symplectic vector space of dimension $2n$, and $\kh_V:=V\oplus \gR E$ be the corresponding Heisenberg algebra : $[x,y]=\omega_0(x,y)E$, $[x,E]=0$ $\forall x,y \in V$. Setting $\ka := \gR H$, we consider the split extension of Lie algebras: 
\begin{equation*}
0\rightarrow \kh_V \rightarrow \ks:= \ka \ltimes \kh_V \rightarrow \ka \rightarrow 0,
\end{equation*}
with extension homomorphism $\rho_{\mathfrak{h}} : \mathfrak{a} \rightarrow  \Der(\mathfrak{h})$ given by 
\begin{equation*}
\rho_{\mathfrak{h}}(H)(x+\ell E):=[H,x+\ell E]:=x+2\ell E, \quad x\in V,\, \ell\in \gR. 
\label{defelementarynormal}
\end{equation*}
Then the Lie algebra $\ks$ underlines an elementary normal $j$-algebra. Moreover, 
every elementary normal $j$-algebra is of that form.
\end{Lemma}
\noindent 
The main interest of elementary normal $j$-algebras is that they are the only building blocks of normal $j$-algebras, as shown by the following important property \cite{PS}.
\begin{proposition}\label{PS}
\begin{enumerate}
\item\label{PS_i} Let $(\b,\alpha,j)$ be a normal $j$-algebra, and $\z_1$ be one-dimensional ideal of $\b$. Then there exists a vector subspace $V$ of $\b$, such that $\mathfrak{s}=j\z_1+V+\z_1$ underlies an elementary normal $j$-ideal of $\b$. Moreover,
the associated extension sequence
$$
0\longrightarrow\mathfrak{s}\longrightarrow\b\longrightarrow\b'\longrightarrow0\;,
$$
is split as a sequence of normal $j$-algebras and such that:
\begin{enumerate}
\item $[\b',\a_1\oplus\z_1]=0$,
\item $[\b', V]\subset V$.
\end{enumerate}
\item Such one-dimensional ideals $\z_1$ always exist. In particular, every normal $j$-algebra admits a decomposition as a sequence of split extensions of elementary normal $j$-algebras with properties (a) and (b) above. 
\end{enumerate}
\end{proposition}

\subsection{Symplectic symmetric space geometry of elementary normal $j$-groups}
\label{subsec-elem}
In this section, we briefly recall results of \cite{BM, B07}.
\begin{definition}
The connected simply-connected real Lie group $G$ whose Lie algebra $\kg$ underlies a normal $j$-algebra is called a {\defin normal $j$-group}.
The connected simply connected Lie group $\gS$ whose Lie algebra $\ks$ underlies an elementary normal $j$-algebra is said to be an {\defin elementary normal $j$-group}.
\end{definition}
\noindent Elementary normal $j$-groups are exponential (non-nilpotent) solvable Lie groups. As an example, consider the Lie algebra $\mathfrak{s}$ of Definition \ref{def-elemjalg} where $V=\algzero$. It is generated over $\mathbb{R}$ by two elements $H$ and $E$ satisfying $[H,E]=2E$ and is therefore isomorphic to the Lie algebra of the group of affine transformations of the real line: in this case, $\mathbb{S}$ is the $ax+b$ group.

\noindent Now generally,  the Iwasawa factor $AN$ of the simple group $SU(1,n)$ (which corresponds to the above example in the case $n=1$) is an elementary normal $j$-group. 
\medskip

\noindent We realize $\gS$ on the product manifold underlying $\ks$:
$$
\gS\;=\;\gR\times V\times\gR\;=\;\{(a,x,\ell)\}\;.
$$
The group law of $\gS$ is given by
\begin{equation}
(a,x,\ell)\fois(a',x',\ell')=\Big(a+a',e^{-a'}x+x',e^{-2a'}\ell+\ell'+\frac12e^{-a'}\omega_0(x,x')\Big)\label{eq-gSlaw}
\end{equation}
and the inverse by $$(a,x,\ell)^{-1}=(-a,-e^ax,-e^{2a}\ell)\;.$$

\noindent We denote by
$$
\Ad^*:\gS\times\ks^\ast:(g,\xi)\mapsto\Ad^*_g(\xi)\;:=\;\xi\circ\Ad_{g^{-1}}
$$
the co-adjoint action of $\gS$ on the dual space $\ks^\ast$ of $\ks=\gR H\oplus V\oplus\gR E$.

\noindent In the dual $\ks^\ast$, we consider the elements  ${}^\flat H$ and ${}^\flat E$ as well as ${}^\flat x$ ($x\in V$) defined by 

\begin{align*}
&{}^\flat H|_{V\oplus \gR E}\;\equiv\;0,\qquad \langle{}^\flat H, H\rangle\;=\;1,\\
&{}^\flat E|_{\gR H\oplus V}\;\equiv\;0,\qquad \langle{}^\flat E, E\rangle\;=\;1,\\
&{}^\flat x|_{\gR H\oplus\gR E}\;\equiv\;0,\qquad \langle{}^\flat x,y\rangle\;=\;\omega_0(x,y) \quad (y\in V)\;.
\end{align*}
\begin{proposition}\label{elem-orbit}
\noindent Let $\caO_\ee$ denote the co-adjoint orbit  through the element $\ee\,{}^\flat E$, for $\ee=\pm 1$, equipped with its standard Kirillov-Kostant-Souriau symplectic
structure (referred to as KKS). Then the map
\begin{equation}
\gS\to\caO_\ee\;:\;(a,x,\ell)\mapsto \Ad^*_{(a,x,\ell)}(\ee\,{}^\flat E)=\ee(2\ell\, {}^\flat H-e^{-a}\,{}^\flat x+e^{-2a}\,{}^\flat E) \label{eq-orbitelem}
\end{equation}
is a $\gS$-equivariant global Darboux chart on $\caO_\ee$ in which the KKS two-form reads:
$$\omega:=\omega_{\gS}:=\ee(2\dd a\wedge\dd\ell+\omega_0)\;.$$
\end{proposition}

\noindent Within this setting, we consider the moment map of the action of $\gS$ on $\caO_\ee\simeq\gS$:
$$
\lambda:\ks\to C^\infty(\gS): X\mapsto\lambda_X
$$ 
defined by the relations:
$$
\lambda_X(g)\;:=\;\langle\,\Ad_g^\ast\left(\ee\,{}^\flat E\right)\,,\,X\,\rangle\;.
$$
\begin{lemma} Denoting, for every $X\in\ks$, the associated  fundamental vector field by
$$
X^*_g\;:=\;\frac{\rm d}{{\rm d}t}|_0\,\exp(-tX).g\;,
$$ 
one has ($y\in V$):
\begin{equation*}
H^\ast=-\partial_a\;,\qquad y^\ast=-e^{-a}\partial_{y}+\frac12e^{-a}\omega_0(x,y)\partial_\ell\;,\qquad E^\ast=-e^{-2a}\partial_\ell\;.
\end{equation*}
Moreover the moment map reads:
\begin{equation}
\lambda_H(a,x,\ell)=2\ee\ell\;,\qquad \lambda_y(a,x,\ell)=e^{-a}\ee\omega_0(y,x)\;,\qquad \lambda_E(a,x,\ell)=\ee e^{-2a}\;.\label{eq-elem-moment}
\end{equation}
\end{lemma}

\begin{proposition}
The map
$$
s:\gS\times\gS\to\gS:(g,g')\mapsto s_gg'
$$
defined by
\begin{equation}
s_{(a,x,\ell)}(a',x',\ell')=\Big(2a-a',2\cosh(a-a')x-x',2\cosh(2(a-a'))\ell-\ell'+\sinh(a-a')\omega_0(x,x')\Big)\label{eq-symstruct}
\end{equation}
endows the Lie group $\gS$ with a left-invariant structure of symmetric space in the sense of O. Loos (cf. \cite{L}).

\noindent  Moreover the symplectic structure $\omega$ is invariant under the symmetries: for every $g\in\gS$, one has
$$
s_g^\ast\omega\;=\;\omega\;.
$$
\end{proposition}
\subsection{Normal j-groups}
\label{subsec-normal}

The above Proposition \ref{PS} implies that every normal $j$-group $G$ can be decomposed into  a semi-direct product 
\begin{equation}\label{SDP}
G\;=\;G_1\ltimes_\rho\gS_2
\end{equation}
where $$\gS_2\;:=\;\R H_2\times V_2\times \R E_2$$ is an elementary normal $j$-group of real dimension $2n_2+2$ and  $G_1$ is a normal $j$-group. This means that the group law of $G$ has the form
\begin{equation*}
\forall g_1,g_1'\in G_1,\ \forall g_2,g_2'\in\gS_2\quad:\quad (g_1,g_2)\fois(g_1',g_2')=\Big( g_1\fois g_1',\ g_2\fois(\rho(g_1)g_2')\Big),
\end{equation*}
where $\rho:G_1\to Sp(V_2,\omega_0)$ denotes the extension homomorphism; and the inverse is given by $(g_1,g_2)^{-1}=(g_1^{-1},\rho(g_1^{-1})g_2^{-1})$. As a consequence, every normal $j$-group  therefore results as a sequence of semidirect products of a finite number of elementary normal $j$-groups.

\begin{proposition}
\label{prop-moment} Consider the decomposition (\ref{SDP}). Then, 
\begin{enumerate}
\item the Lie group $G_1$ admits an open co-adjoint orbit $\caO_1$ through an element $o_1\in\kg_1^\ast$ which it acts on in a  simply transitive way.
\item The co-adjoint   orbit $\caO$ of $G$ through the element $o\;:=\;o_1+\ee_2\,{}^\flat E_2$ (same notation as in Section \ref{subsec-elem})
is open in $\kg^\ast$.
\item Denoting by $\caO_2$ the co-adjoint orbit of $\gS_2$ through $\ee_2\,{}^\flat E_2$, the map
\begin{equation}\label{SYMPL}
\phi\,:\,\caO_1\times\caO_2\to\caO:(\Ad_{g_1}^\ast o_1,\ee_2\Ad_{g_2}^\ast{}^\flat E_2)\mapsto\Ad^\ast_{(g_1,g_2)}(o)
\end{equation}
is a symplectomorphism when endowing each orbit with its KKS two-form.
\end{enumerate}
\end{proposition}
\begin{proof} We proceed by induction on the dimension in proving that $\caO$ is acted on by $G$ in a simply transitive way.
By induction hypothesis, so is $\caO_1$ by $G_1$. And Proposition \ref{elem-orbit} implies it is the case for $\caO_2$ as well.
Now denoting $(g_1,e)\;=:\;g_1$ and $(e,g_2)\;=:\;g_2$, we observe:
$$
\Ad^\ast_{(g_1,g_2)}(o)=\Ad^\ast_{g_2g_1}(o)=\Ad^\ast_{g_2}\left(\Ad^\ast_{g_1}(o_1)+\ee_2\,{}^\flat E_2\circ\rho(g_1^{-1})_{\ast e}\right)
$$
where $\rho:G_1\to\Aut(\gS_2)$ denotes the extension homomorphism.

\noindent Now for all $\xi_1\in\kg_1^\ast\,,\, X_1\in\kg_1\,,\,X_2\in\ks_2$ and $g_2\in\gS_2$:
$$
\langle\Ad_{g_2}^\ast\xi_1\,,\,X_1+X_2\rangle=\langle\xi_1\,,\,\Ad_{g_2^{-1}}X_1+\Ad_{g_2^{-1}}X_2\rangle=\langle\xi_1\,,\,\Ad_{g_2^{-1}}X_1\rangle\;.
$$
But
$$
\Ad_{g_2^{-1}}X_1\;=\;\frac{{\rm d}}{{\rm d}t}|_0(\exp(t X_1),g_2^{-1})(e,g_2)\;=\;\frac{{\rm d}}{{\rm d}t}|_0(\exp(t X_1),g_2^{-1}\rho(\exp(tX_1))g_2)\;.
$$
Hence
$$
\langle\xi_1\,,\,\Ad_{g_2^{-1}}X_1\rangle\;=\;\langle\xi_1\,,\,X_1\oplus\left(\frac{{\rm d}}{{\rm d}t}|_0g_2^{-1}\rho(\exp(tX_1))g_2\right)\rangle\;=\;\langle\xi_1\,,\,X_1\rangle\;.
$$
Therefore $\Ad_{g_2}^\ast\xi_1=\xi_1$ and we get
$$
\Ad^\ast_{(g_1,g_2)}(o)\;=\;\Ad^\ast_{g_1}(o_1)\,+\ee_2\,\Ad^\ast_{g_2}\left({}^\flat E_2\circ\rho(g_1^{-1})_{\ast e}\right)\;.
$$
The induction hypothesis thus implies that the stabilizer of element $o$ in $G$ is trivial. Which shows in particular that  the fundamental group
of $\caO$ is trivial.
The map (\ref{SYMPL}) being a surjective submersion is therefore a diffeomorphism. 

\noindent It remains to prove the assertion regarding the symplectic structures. Denoting by $\omega^{\caO}$ the KKS form on $\caO$, we observe that,
with obvious notations, 
for all $Y_1\in\kg_1$ and $Y_2\in\ks_2$:
\begin{eqnarray*}
&&\phi^\ast\omega^{\caO}(X^\ast_1\oplus X_2^\ast,Y_1^\ast\oplus Y_2^\ast)\\&=&
\omega^{\caO}_{\Ad^\ast_{(g_1,g_2)}(o)}(\phi_\ast X^\ast_1+ \phi_\ast  X_2^\ast,\phi_\ast Y_1^\ast+ \phi_\ast Y_2^\ast)\\
&=&
\omega^{\caO}_{\Ad^\ast_{(g_1,g_2)}(o)}(\left(\Ad_{g_2}X_1\right)^\ast+  X_2^\ast,\left(\Ad_{g_2}Y_1\right)^\ast+ Y_2^\ast)\\
&=&
\langle\,\Ad^\ast_{(g_1,g_2)}(o)\,,\,[\Ad_{g_2}X_1+  X_2\,,\,\Ad_{g_2}Y_1+ Y_2]\,\rangle\\
&=&
\langle\,\Ad^\ast_{g_1}(o)\,,\,[X_1+  \Ad_{g_2^{-1}}X_2\,,\,Y_1+  \Ad_{g_2^{-1}}Y_2]\,\rangle\\
&=&
\langle\,\Ad^\ast_{g_1}(o)\,,\,[X_1,Y_1]-  \rho(Y_1)\Ad_{g_2^{-1}}X_2+  \rho(X_1)\Ad_{g_2^{-1}}Y_2+\Ad_{g_2^{-1}}[X_2,Y_2]\,\rangle\\
&=&\omega^{\caO_1}_{\Ad^\ast_{g_1}(o_1)}(X^\ast_1,Y_1^\ast)+\ee_2\omega^{\caO_2}_{\Ad^\ast_{g_2}{}^\flat E_2}(X^\ast_2,Y_2^\ast)\\
&+&\ee_2\langle{}^\flat E_2\,,\,\rho(g_1^{-1})_{\ast e}\left(
-  \rho(Y_1)\Ad_{g_2^{-1}}X_2+  \rho(X_1)\Ad_{g_2^{-1}}Y_2
\right)\rangle\;.
\end{eqnarray*}
The last term in the above expression vanishes identically. Indeed, the specific form of $\rho$ implies that the element
$v_2\;:=\;-  \rho(Y_1)\Ad_{g_2^{-1}}X_2+  \rho(X_1)\Ad_{g_2^{-1}}Y_2$ lives in $V_2$ as well as $\rho(g_1^{-1})_{\ast e}v_2$.
\end{proof}
 
 \begin{remark}
 \label{rmk-decomp}
 Normal $j$-groups can be decomposed into elementary normal $j$-groups $\gS_k$ as $G=\big(\dots (\gS_1\ltimes_{\rho_1}\gS_2)\ltimes_{\rho_2}\dots\big)\ltimes_{\rho_{N-1}}\gS_N$ and the co-adjoint orbits described in Proposition \ref{prop-moment} are determined by sign choices $\ee_k=\pm 1$ for each factor $\gS_k$. We will denote $\caO_{(\ee)}$ the co-adjoint orbit associated to the signs $(\ee_k)_{1\leq k\leq N}\in(\gZ_2)^N$.
 \end{remark}

\begin{example}
\label{ex-normal-siegel}
Let us describe the following example corresponding to the six-dimensional Siegel domain $\mbox{\rm Sp}(2,\R)/\mbox{\rm U}(2)$. Let $G_1=\gS_1$ be of dimension 2 ($V_1=\algzero$, $G_1$ is the affine group), $\gS_2$ of dimension 4, i.e. $V_2$ is of dimension 2, with basis $f_2,f'_2$ endowed with $\omega_0=\begin{pmatrix} 0&1 \\ -1& 0 \end{pmatrix}$), and the action $\rho:\gS_1\to Sp(V_2)$ be given by
\begin{equation*}
\rho(a_1,\ell_1)=\begin{pmatrix} e^{a_1}& 0 \\ e^{-a_1}\ell_1 & e^{-a_1} \end{pmatrix}.
\end{equation*}
Then the group law is
\begin{multline*}
(a_1,\ell_1,a_2,v_2,w_2,\ell_2)\fois(a'_1,\ell'_1,a_2',v_2',w_2'\ell_2')=\Big(a_1+a_1',e^{-2a_1'}\ell_1+\ell_1',a_2+a_2',e^{-a_2'}v_2+e^{a_1}v_2',\\
e^{-a_2'}w_2+e^{-a_1}\ell_1v_2'+e^{-a_1}w_2', e^{-2a_2'}\ell_2+\ell_2'+\frac12 e^{-a_2'}(e^{-a_1}\ell_1 v_2v_2'+e^{-a_1}v_2w_2'-e^{a_1}w_2v_2')\Big),
\end{multline*}
where 
$(a_1,\ell_1)\in\gS_1\;,\;(a_2,v_2,w_2,\ell_2)\in\gS_2$ and
$$
g:=(a_1,\ell_1,a_2,v_2,w_2,\ell_2)=e^{a_2H_2}e^{v_2f_2+w_2f'_2}e^{\ell_2E_2}e^{a_1H_1}e^{\ell_1 E_1}\;.$$
Its Lie algebra is characterized by
\begin{align*}
&[H_1,E_1]=2E_1,\qquad [H_2,f_2]=f_2,\qquad [H_2,f'_2]=f'_2,\qquad [f_2,f'_2]=E_2,\\
&[H_2,E_2]=2E_2,\qquad [H_1,f_2]=f_2,\qquad [H_1,f'_2]=-f'_2,\qquad [E_1,f_2]=f'_2,
\end{align*}
where the other relations vanish. The co-adjoint action takes the form
\begin{multline*}
\Ad^{*}_{g}(\ee_1\,{}^\flat E_1+\ee_2\,{}^\flat E_2)=(2\ee_1\ell_1+\ee_2 v_2w_2){}^\flat H_1+(\ee_1 e^{-2a_1}-\frac{\ee_2}{2}v_2^2){}^\flat E_1\\
+\ee_2(2\ell_2{}^\flat H_2-e^{-a_2}v_2{}^\flat f_2-e^{-a_2}w_2{}^\flat f'_2+e^{-2a_2}{}^\flat E_2)\;.
\end{multline*}
The moment map can then be extracted from this expression:
\begin{align*}
&\lambda_{H_1}=2\ee_1\ell_1+\ee_2 v_2w_2,\qquad \lambda_{E_1}=\ee_1 e^{-2a_1}-\frac{\ee_2}{2}v_2^2,\qquad \lambda_{H_2}=2\ee_2\ell_2,\\
&\lambda_{f_2}=\ee_2 e^{-a_2}w_2,\qquad \lambda_{f'_2}=-\ee_2 e^{-a_2}v_2,\qquad \lambda_{E_2}=\ee_2 e^{-2a_2}.
\end{align*}
\end{example}

\section{Determination of the star-exponential}

\subsection{Quantization of elementary groups}
\label{subsec-quelem}

We follow the analysis developed in \cite{BG}, where the reader can find all the proofs. In the notations of section \ref{subsec-elem}, we choose two Lagrangian subspaces in duality $V_0,V_1$ of the symplectic vector space $(V,\omega_0)$ of dimension $2n$ underlying the elementary group $\gS$. We denote the corresponding coordinates $x=(v,w)\in V$ in the global chart, with $v\in V_0$ and $w\in V_1$. Let $\kq=\gR H\oplus V_0$ and $Q=\exp(\kq)$. The unitary induced representation associated to the co-adjoint orbit $\caO_\ee$ ($\ee=\pm1$) by the method of Kirillov has the form:
\begin{equation}
U_{\theta,\ee}(a,x,\ell)\varphi(a_0,v_0)= e^{\frac{i\ee}{\theta}\Big(e^{2(a-a_0)}\ell+\omega_0(\frac12e^{a-a_0}v-v_0,e^{a-a_0}w)\Big)}\varphi(a_0-a,v_0-e^{a-a_0}v)\label{eq-unitindrep}
\end{equation}
for $(a,x,\ell)\in \gS$, $\varphi\in L^2(Q)$, $(a_0,v_0)\in Q$ and $\theta\in\gR_+^*$. These representations $U_{\theta,\ee}:\gS\to\caL(\ehH)$ are unitary and irreducible, and the unitary dual is described by these two representations. A multiplier $\bfm$ is a function on $Q$. There is a particular multiplier:
\begin{equation}
\bfm_0(a,v)=2^{n+1}\sqrt{\cosh(2a)}\cosh(a)^n.\label{eq-multiplier}
\end{equation}
Let us define $\Sigma:=(s_{(0,0,0)}|_Q)^\ast$, where $s$ is the symmetric structure \eqref{eq-symstruct}:
\begin{equation}
\Sigma\varphi(a,v)=\varphi(-a,-v).\label{eq-sigma}
\end{equation}
Then, the Weyl-type quantization map is given by
\begin{multline}
\Omega_{\theta,\ee,\bfm_0}(a,x,\ell)\varphi(a_0,v_0):=U_{\theta,\ee}(a,x,\ell)\bfm_0\Sigma U_{\theta,\ee}(a,x,\ell)^{-1}\varphi(a_0,v_0)\\
=2^{n+1}\sqrt{\cosh(2(a-a_0))}\cosh(a-a_0)^n e^{\frac{2i\ee}{\theta}\Big(\sinh(2(a-a_0))\ell+\omega_0(\cosh(a-a_0)v-v_0,\cosh(a-a_0)w)\Big)}\\
\varphi(2a-a_0,2\cosh(a-a_0)v-v_0).\label{eq-qumap}
\end{multline}
$\Omega_{\theta,\ee,\bfm_0}(g)$ is a symmetric unbounded operator on $\ehH$, and $g\in \gS\simeq\caO_\ee$.

On smooth functions with compact support $f\in\caD(\caO_\ee)$, by denoting $\kappa:=\frac{1}{2^n(\pi\theta)^{n+1}}$, one has
\begin{equation*}
\Omega_{\theta,\ee,\bfm_0}(f):=\kappa\int_{\caO_\ee} f(g)\Omega_{\theta,\ee,\bfm_0}(g)\dd\mu(g)
\end{equation*}
with $\dd\mu(g)=\dd^Lg$ which corresponds to the Liouville measure of the KKS symplectic form on the co-adjoint orbit $\caO_\ee\simeq\gS$. Its extension is continuous and called the quantization map $\Omega_{\theta,\ee,\bfm_0}:L^2(\caO_\ee)\to\caL_{HS}(\ehH)$, with $\ehH:=L^2(Q)$ and $\caL_{HS}$ the Hilbert-Schmidt operators.
The normalization has been chosen such that $\Omega_{\theta,\ee,\bfm_0}(1)=\gone_\ehH$, understood in the distributional sense. Moreover, it is $\gS$-equivariant, because of
\begin{equation*}
\forall g,g_0\in \gS\quad:\quad \Omega_{\theta,\ee,\bfm_0}(g\fois g_0)=U_{\theta,\ee}(g)\Omega_{\theta,\ee,\bfm_0}(g_0)U_{\theta,\ee}(g)^{-1}.
\end{equation*}

The unitary representation $U_{\theta,\ee}:\gS\to\caL(\ehH)$ induces a resolution of the identity.
\begin{proposition}
By denoting the norm $\norm\varphi\norm^2_w:=\int_Q|\varphi(a,v)|^2e^{2(n+1)a}\dd a\dd v$ and $\varphi_g(q)=U_{\theta,\ee}(g)\varphi(q)$ for $g\in\gS$, $q\in Q$ and a non-zero $\varphi\in\ehH$, we have
\begin{equation*}
\frac{\kappa}{\norm\varphi\norm^2_w}\int_\gS |\varphi_g\rangle\langle\varphi_g|\dd^Lg=\gone_\ehH.
\end{equation*}
\end{proposition}
This resolution of identity shows that the trace has the form
\begin{equation}
\tr(T)=\frac{\kappa}{\norm\varphi\norm^2_w}\int_\gS \langle\varphi_g, T\varphi_g\rangle\dd^L g\label{eq-tracesymb}
\end{equation}
for any trace-class operator $T\in\caL^1(\ehH)$.
\begin{theorem}
The symbol map, which is the left-inverse of the quantization map $\Omega_{\theta,\ee,\bfm_0}$ can be obtained via the formula
\begin{equation*}
\forall f\in L^2(\caO_\ee),\ \forall g\in\caO_\ee\quad:\quad \tr(\Omega_{\theta,\ee,\bfm_0}(f)\Omega_{\theta,\ee,\bfm_0}(g))=f(g),
\end{equation*}
where the trace is understood in the distributional sense in the variable $g\in \gS$.
\end{theorem}

Then, the star-product is defined as
\begin{equation*}
(f_1\star_{\theta,} f_2)(g):=\tr(\Omega_{\theta,\ee,\bfm_0}(f_1)\Omega_{\theta,\ee,\bfm_0}(f_2)\Omega_{\theta,\ee,\bfm_0}(g))
\end{equation*}
for $f_1,f_2\in L^2(\caO_\ee)$ and $g\in\caO_\ee$, where we omitted the subscripts $\ee,\bfm_0$ for the star-product.

\begin{proposition}
\label{prop-quelem-starprod}
The star-product has the following expression:
\begin{equation}
(f_1\star_{\theta}f_2)(g)=\frac{1}{(\pi\theta)^{2n+2}}\int\ K_{\gS}(g,g_1,g_2)e^{-\frac{2i}{\theta}S_{\gS}(g,g_1,g_2)}f_1(g_1)f_2(g_2)\dd \mu(g_1)\dd \mu(g_2)\label{eq-starprod}
\end{equation}
where the amplitude and the phase are
\begin{align*}
K_{\gS}(g,g_1,g_2)=& 4\sqrt{\cosh(2(a_1-a_2))\cosh(2(a_1-a))\cosh(2(a-a_2))}\cosh(a_2-a)^{n}\\
&\cosh(a_1-a)^{n}\cosh(a_1-a_2)^{n},\\
\ee S_{\gS}(g,g_1,g_2)=&-\sinh(2(a_1-a_2))\ell-\sinh(2(a_2-a))\ell_1-\sinh(2(a-a_1))\ell_2\\
&+\cosh(a_1-a)\cosh(a_2-a)\omega_0(x_1,x_2) +\cosh(a_1-a)\cosh(a_1-a_2)\omega_0(x_2,x)\\
&+\cosh(a_1-a_2)\cosh(a_2-a)\omega_0(x,x_1).
\end{align*}
with $g_i=(a_i,x_i,\ell_i)\in\gS$. Moreover, $g\mapsto 1$ is the unit of this product, is is associative, $\gS$-invariant and satisfies the tracial identity:
\begin{equation}
\int f_1\star_\theta f_2=\int f_1\fois f_2.\label{eq-tracial1}
\end{equation}
\end{proposition}

Note that this product has first been found \cite{Bieliavsky:2002} by intertwining the Moyal product:
\begin{multline*}
(f_1\star^0_\theta f_2)(a,x,\ell)=\frac{4}{(\pi\theta)^{2+2n}}\int\dd a_i\dd x_i\dd \ell_i\ f_1(a_1+a,x_1+x,\ell_1+\ell)\\
f_2(a_2+a,x_2+x,\ell_2+\ell) e^{-\frac{2i\ee}{\theta}(2a_1\ell_2-2a_2\ell_1+\omega_0(x_1,x_2))}
\end{multline*}
for $f_1,f_2\in L^2(\caO_\ee)$ and $\caO_\ee\simeq\gS\simeq\gR^{2n+2}$, which is $\ks$-covariant ($[\lambda_X,\lambda_Y]_{\star_\theta^0}=-i\theta \lambda_{[X,Y]}$) but not $\gS$-invariant. So for smooth functions with compact support, we have $f_1\star_{\theta}f_2= T_\theta((T_\theta^{-1}f_1)\star_\theta^0(T_\theta^{-1}f_2))$ with intertwiners:
\begin{align}
&T_\theta f(a,x,\ell)=\frac{1}{2\pi}\int\dd t\dd\xi\ \sqrt{\cosh(\frac{\theta t}{2})}\cosh(\frac{\theta t}{4})^n e^{\frac{2i}{\theta}\sinh(\frac{\theta t}{2})\ell-i\xi t}f(a,\cosh(\frac{\theta t}{4})x,\xi)\nonumber\\
&T^{-1}_\theta f(a,x,\ell)=\frac{1}{2\pi}\int\dd t\dd\xi\ \frac{\sqrt{\cosh(\frac{\theta t}{2})}}{\cosh(\frac{\theta t}{4})^n}e^{-\frac{2i}{\theta}\sinh(\frac{\theta t}{2})\xi+it\ell}f(a,\cosh(\frac{\theta t}{4})^{-1}x,\xi)\label{eq-quelem-intertw}
\end{align}
which will be useful in section \ref{subsec-other}.

\subsection{Quantization of normal j-groups}
\label{subsec-qunorm}

Let $G=G_1\ltimes \gS_2$ be a normal $j$-group, with notations as in section \ref{subsec-normal}. Taking into account its structure, the unitary representation $U$ and the quantization map $\Omega$ of this group (dependence in $\theta\in\gR_+^*$ will be omitted here in the subscripts) can be constructed from the ones $U_1$ and $\Omega_1$ of $G_1$ (obtained by recurrence) and the ones $U_2$ and $\Omega_2$ of $\gS_2$, given by \eqref{eq-unitindrep} and \eqref{eq-qumap} (without $\bfm_0$ for the moment).

Let $\ehH_i$ be the Hilbert space of the representation $U_i$, associated to a co-adjoint orbit $\caO_i$ (in the notations of Proposition \ref{prop-moment}). Since $U_2$ is irreducible and $\rho:G_1\to Sp(V_2)$, there exists a unique homomorphism $\caR:G_1\to\caL(\ehH_2)$ such that $\forall g_1\in G_1$, $\forall g_2\in \gS_2$,
\begin{equation*}
U_2(\rho(g_1)g_2)=\caR(g_1)U_2(g_2)\caR(g_1)^{-1}.
\end{equation*}
$\caR$ is actually a metaplectic-type representation associated to $U_2$ and $\rho$. The matrix $\rho(g_1)$, with smooth coefficients in $g_1$, is of the form
\begin{equation*}
\rho(g_1)=\begin{pmatrix} \rho_+(g_1)&0 \\ \rho_-(g_1)& (\rho_+(g_1)^T)^{-1} \end{pmatrix}
\end{equation*}
with $\rho_-(g_1)^T\rho_+(g_1)=\rho_+(g_1)^T\rho_-(g_1)$.
\begin{proposition}
\label{prop-metaplec}
The map $\caR:G_1\to\caL(\ehH_2)$ is given by $\forall g_1\in G_1$, $\forall\varphi\in\ehH_2$ non-zero,
\begin{equation*}
\caR(g_1)\varphi(a_0,v_0)=\frac{1}{|\det(\rho_+(g_1))|^{\frac12}}e^{-\frac{i\ee_2}{2\theta}v_0\rho_-(g_1)\rho_+(g_1)^{-1}v_0}\varphi(a_0,\rho_+(g_1)^{-1}v_0)
\end{equation*}
and is unitary, where the sign $\ee_2=\pm1$ determines the choice of the co-adjoint orbit $\caO_2$ and the associated irreducible representation $U_2$.
\end{proposition}
The expression
\begin{equation*}
U(g)\varphi:=U_1(g_1)\varphi_1\,\otimes\,U_2(g_2)\caR(g_1)\varphi_2
\end{equation*}
for $g=(g_1,g_2)\in G$, $\varphi=\varphi_1\otimes\varphi_2\in\ehH:=\ehH_1\otimes\ehH_2$, defines a unitary representation $U:G\to\caL(\ehH)$. Let $\Sigma=\Sigma_1\otimes\Sigma_2$, with $\Sigma_2$ given in \eqref{eq-sigma}. Then, the quantization map is defined as
\begin{equation*}
\Omega(g):=U(g)\circ \Sigma\circ U(g)^{-1}.
\end{equation*}
Using the definition of $U$ and $\caR$ together with the property (see Proposition 6.55 in \cite{BG}):
\begin{equation*}
\caR(g_1)\Sigma_2\caR(g_1)^{-1}=\Sigma_2,
\end{equation*}
it is easy to check that
\begin{equation*}
\Omega((g_1,g_2))=\Omega_1(g_1)\otimes\Omega_2(g_2),
\end{equation*}
with $(g_1,g_2)\in G$ and $\Omega_i(g_i)=U_i(g_i)\circ \Sigma_i\circ U_i(g_i)^{-1}$. Using the identification $\caO\simeq G$ (see Proposition \ref{prop-moment}), we see that $\Omega$ is defined on $\caO$ and it is again $G$-equivariant: for $g\in G$ and $g'\in \caO$,
\begin{equation*}
\Omega(g\fois g')=U(g)\Omega(g')U(g)^{-1}.
\end{equation*}
In the same way, if $\bfm_0:=\bfm_0^1\otimes\bfm_0^2$, where $\bfm_0^2$ is given by \eqref{eq-multiplier}, we also have $\Omega_{\bfm_0}((g_1,g_2))=\Omega_{1,\bfm_0^1}(g_1)\otimes \Omega_{2,\bfm_0^2}(g_2)$. The quantization map of functions $f\in \caD(\caO)$ has then the form
\begin{equation*}
\Omega_{\bfm_0}(f):=\kappa\int_\caO f(g)\Omega_{\bfm_0}(g)\dd\mu(g)
\end{equation*}
where $\dd\mu(g):=\dd\mu_1(g_1)\dd\mu_2(g_2)=\dd^Lg_1\dd^Lg_2$ is the Liouville measure of the KKS symplectic form on the co-adjoint orbit $\caO\simeq G$; $\kappa=\kappa_1\kappa_2$, for $G=G_1\ltimes\gS_2$, is defined recursively with $\kappa_2=\frac{1}{2^{n_2}(\pi\theta)^{n_2+1}}$ and $\dim(\gS_2)=2n_2+2$ like in section \ref{subsec-quelem}. We then have $\Omega_{\bfm_0}(1)=\gone$.

Note that the left-invariant measure for the group $G=G_1\ltimes \gS_2$ has the form
\begin{equation*}
\dd^Lg=\dd^Lg_1\dd^Lg_2
\end{equation*}
which corresponds to the Liouville measure $\dd\mu(g)$, like the elementary case. As in section \ref{subsec-quelem}, the unitary representation $U:G\to\caL(\ehH)$ induces a resolution of the identity.
\begin{proposition}
\label{prop-qunorm-resol}
By denoting the norm $\norm\varphi\norm^2_w:=\norm\varphi_1\norm^2_w \norm\varphi_2\norm^2_w$ for $\varphi=\varphi_1\otimes\varphi_2\in\ehH$ non-zero, we have
\begin{equation*}
\frac{\kappa}{\norm\varphi\norm^2_w}\int_G |U(g)\varphi\rangle\langle U(g)\varphi|\dd^Lg=\gone_\ehH.
\end{equation*}
\end{proposition}
This resolution of identity shows that the trace has the form
\begin{equation}
\tr(T)=\frac{\kappa}{\norm\varphi\norm^2_w}\int_G \langle U(g)\varphi, TU(g)\varphi\rangle\dd^L g\label{eq-tracesymbnorm}
\end{equation}
for $T\in\caL^1(\ehH)$. In particular, for $T=T_1\otimes T_2$ with $T_i\in\caL^1(\ehH_i)$, one has
\begin{multline}
\tr(T)=\frac{\kappa}{\norm\varphi\norm^2_w}\int_G \langle U_1(g_1)\varphi_1, T_1U_1(g_1)\varphi_1\rangle \langle U_2(g_2)\caR(g_1)\varphi_2, T_2U_2(g_2)\caR(g_1)\varphi_2\rangle\\
\dd^L g_1\dd^Lg_2=\tr(T_1)\tr(T_2).\label{eq-qunorm-trace}
\end{multline}
\begin{theorem}
The symbol map, which is the left-inverse of the quantization map $\Omega_{\bfm_0}$ can be obtained via the formula
\begin{equation*}
\forall f\in L^2(\caO),\ \forall g\in \caO\quad:\quad \tr\Big(\Omega_{\bfm_0}(f)\Omega_{\bfm_0}(g)\Big)=f(g).
\end{equation*}
\end{theorem}
\begin{proof}
Abstractly (in a weak sense), we have
\begin{multline*}
\tr(\Omega_{\bfm_0}(f)\Omega_{\bfm_0}(g))=\kappa\int f(g_1',g_2')\tr(\Omega_1(g_1')\Omega_1(g_1)) \tr(\Omega_2(g_2')\Omega_2(g_2))\dd\mu_1(g_1')\dd\mu_2(g_2')\\
=f(g_1,g_2).
\end{multline*}
\end{proof}

The star-product is defined as
\begin{equation*}
(f_1\star_{\theta} f_2)(g):=\tr(\Omega_{\bfm_0}(f_1)\Omega_{\bfm_0}(f_2)\Omega_{\bfm_0}(g))
\end{equation*}
for $f_1,f_2\in L^2(\caO)$ and $g\in \caO$.
\begin{proposition}
\label{prop-qunorm-starprod}
The star-product has the following expression:
\begin{equation}
(f_1\star_{\theta}f_2)(g)=\frac{1}{(\pi\theta)^{\dim(G)}}\int_{G\times G}\ K_{G}(g,g',g'')e^{-\frac{2i}{\theta}S_{G}(g,g',g'')}f_1(g')f_2(g'')\dd \mu(g')\dd \mu(g'')\label{eq-starprodnorm}
\end{equation}
where the amplitude and the phase are
\begin{align*}
K_{G}(g,g',g'')=& K_{G_1}(g_1,g_1',g_1'')K_{\gS_2}(g_2,g_2',g_2''),\\
S_{G}(g,g',g'')=& S_{G_1}(g_1,g_1',g_1'')+S_{\gS_2}(g_2,g_2',g_2'').
\end{align*}
with $g=(g_1,g_2)\in \caO=\caO_1\times \caO_2$ due to \eqref{SYMPL}. There is also a tracial identity:
\begin{equation*}
\int_\caO (f_1\star_\theta f_2)(g)\dd\mu(g)=\int_\caO f_1(g)f_2(g)\dd\mu(g).
\end{equation*}
\end{proposition}

\subsection{Computation of the star-exponential}
\label{subsec-comput}

\begin{definition}
\label{def-comput-starexp}
We define the {\defin star-exponential} associated to the deformation quantization $(\star,\Omega)$ of section \ref{subsec-qunorm} as
\begin{equation*}
\forall g\in G,\,\forall g'\in \caO\simeq G\quad:\quad \caE^\caO_g(g')=\tr(U(g)\Omega_{\bfm_0}(g'))
\end{equation*}
where the trace has to be understood in the distributional sense in $(g,g')\in G\times \caO$.
\end{definition}

By using computation rules of the above sections, we can obtain recursively on the number of factors of the normal $j$-group $G=G_1\ltimes \gS_2$ with corresponding co-adjoint orbit $\caO\simeq\caO_1\times\caO_2$, the expression of the star-exponential $\caE^\caO\in\caD'(G\times \caO)$.
\begin{theorem}
\label{thm-exprexp}
We have $\forall g,g'\in G$,
\begin{multline}
\caE^\caO_g(g')=\caE^{\caO_1}_{g_1}(g'_1) \frac{2^{n_2}|\det(\rho_+(g_1))|^{\frac12}\sqrt{\cosh(a_2)}\cosh(\frac{a_2}{2})^{n_2}}{|\det(1+\rho_+(g_1))|}\exp(\frac{i\ee_2}{\theta}\Big[ 2\sinh(a_2)\ell'_2\\
+e^{a_2-2a_2'}\ell_2+e^{\frac{a_2}{2}-a_2'}\cosh(\frac{a_2}{2})\omega_0(x_2,x_2')+\frac12(\tilde x)^TM_\rho(g_1) \tilde x\Big]),\label{eq-comput-starexp2}
\end{multline}
where
\begin{equation*}
M_\rho(g_1):=\begin{pmatrix} -B_\rho & C_\rho^T  \\  C_\rho & 0 \end{pmatrix},\qquad
\tilde x:= e^{\frac{a_2}{2}-a_2'}x_2-2\cosh(\frac{a_2}{2})x_2'
\end{equation*}
and with $B_\rho=(1+\rho_+^T(g_1))^{-1}\rho_+^T(g_1)\rho_-(g_1)(1+\rho_+(g_1))^{-1}$, and $C_\rho=\frac12 (\rho_+(g_1)-1)(\rho_+(g_1)+1)^{-1}$, $g=(g_1,g_2)$, $g_2=(a_2,x_2,\ell_2)\in\gS_2$, and $\caE^{\caO_1}$ the star-exponential of the normal $j$-group $G_1$.
\end{theorem}
\begin{proof}
First, we use Proposition \ref{prop-qunorm-resol} and equation \eqref{eq-qunorm-trace}:
\begin{equation*}
\caE^\caO_g(g')=\tr(U(g)\Omega_{\bfm_0}(g'))=\tr(U_1(g_1)\Omega_{1,\bfm_0^1}(g'_1))\tr(U_2(g_2)\caR(g_1)\Omega_{2,\bfm_0^2}(g'_2)).
\end{equation*}
The second part of the above expression can be computed by using \eqref{eq-tracesymb}, it gives:
\begin{equation*}
\tr(U_2(g_2)\caR(g_1)\Omega_{2,\bfm_0^2}(g'_2))= \frac{\kappa_2}{\norm\varphi\norm^2_w}\int_{\gS_2} \langle\varphi_{g_2''}, U_2(g_2)\caR(g_1)\Omega_{2,\bfm_0^2}(g'_2)\varphi_{g_2''}\rangle\dd^Lg_2''
\end{equation*}
If we replace $U_2$, $\Omega_{2,\bfm_0^2}$ and $\caR$ by their expressions determined previously in \eqref{eq-unitindrep}, \eqref{eq-qumap} and Proposition \ref{prop-metaplec}, we find after some integrations and simplifications that $\forall g,g'\in G$,
\begin{equation}
\caE^\caO_g(g')=\caE^{\caO_1}_{g_1}(g'_1) \frac{2^{n_2}|\det(\rho_+(g_1))|^{\frac12}\sqrt{\cosh(a_2)}\cosh(\frac{a_2}{2})^{n_2}}{|\det(1+\rho_+(g_1))|}\exp(\frac{i\ee_2}{\theta}\Big[ 2\sinh(a_2)\ell'_2
+e^{a_2-2a_2'}\ell_2+X^TA_\rho X\Big]).\label{eq-comput-starexp}
\end{equation}
where
\begin{align*}
A_\rho=&\begin{pmatrix} -B_\rho & C_\rho^T & B_\rho & (1+\rho_+^T(g_1))^{-1} \\ C_\rho & 0 & -\rho_+(g_1)(1+\rho_+(g_1))^{-1} & 0\\ B_\rho & -\rho_+^T(g_1)(1+\rho_+^T(g_1))^{-1} & -B_\rho & C_\rho^T \\ (1+\rho_+(g_1))^{-1} & 0 & C_\rho & 0 \end{pmatrix},\\
X=&\begin{pmatrix} \frac{1}{\sqrt 2}e^{\frac{a_2}{2}-a'_2}v_2 \\ \frac{1}{\sqrt 2}e^{\frac{a_2}{2}-a'_2}w_2 \\ \sqrt 2 \cosh(\frac{a_2}{2})v'_2 \\ \sqrt 2 \cosh(\frac{a_2}{2})w'_2\end{pmatrix}
\end{align*}
and with $x_2=(v_2,w_2)$. A straighforward computation then gives the result.
\end{proof}
Let us denote by $\caE^{\caO_2}_{(g_1,g_2)}(g_2')$ the explicit part in the RHS of \eqref{eq-comput-starexp} which corresponds to the star-exponential of the group $\gS_2$ twisted by the action of $g_1\in G_1$.

The expression \eqref{eq-comput-starexp} seems to be ill-defined when $\det(1+\rho_+^{-1})=0$. However, one can obtain in this case a degenerated expression of the star-exponential which is well-defined. For example, when $\rho_+(g_1)=-\gone_{n_2}$, we have
\begin{multline*}
\caE^{\caO_2}_{(g_1,g_2)}(g'_2)= (\pi\theta)^{n_2}\frac{\sqrt{\cosh(a_2)}}{\cosh(\frac{a_2}{2})^{n_2}} \exp(\frac{i\ee_2}{\theta}\Big[ 2\sinh(a_2)\ell'_2
+e^{a_2-2a_2'}\ell_2+\frac12 e^{a_2-2a_2'}\omega_0(v_2,w_2)\Big])\\
\delta\Big(v_2'-\frac{e^{\frac{a_2}{2}-a_2'}}{2\cosh(\frac{a_2}{2})}v_2\Big) \delta\Big(w_2'-\frac{e^{\frac{a_2}{2}-a_2'}}{2\cosh(\frac{a_2}{2})}w_2\Big).
\end{multline*}

In the case where $\rho(g_1)=\gone$, i.e. when the action of $G_1$ on $\gS_2$ is trivial in $G$, we find the second part of the star-exponential
\begin{equation}
\caE^{\caO_2}_{g_2}(g'_2)=\sqrt{\cosh(a_2)}\cosh(\frac{a_2}{2})^{n_2}\exp(\frac{i\ee_2}{\theta}\Big[ 2\sinh(a_2)\ell'_2
+e^{a_2-2a_2'}\ell_2+e^{\frac{a_2}{2}-a_2'}\cosh(\frac{a_2}{2})\omega_0(x_2,x_2')\Big]).\label{eq-comput-starexpelem}
\end{equation}
which corresponds to the star-exponential of the elementary normal $j$-group $\gS_2$.

By using this characterization in terms of the quantization map, we can derive easily some properties of the star-exponential.
\begin{proposition}
\label{prop-link3}
The star-exponential enjoys the following properties. $\forall g,g'\in G$, $\forall g_0\in \caO$,
\begin{itemize}
\item hermiticity: $\overline{\caE^\caO_g(g_0)}=\caE^\caO_{g^{-1}}(g_0)$.
\item covariance: $\caE^\caO_{g'\fois g\fois g'^{-1}}(g'\fois g_0)=\caE^\caO_g(g_0)$.
\item BCH: $\caE^\caO_g\star_\theta\caE^\caO_{g'}=\caE^\caO_{g\fois g'}$.
\item Character formula: $\int_G \caE^\caO_g(g_0)\dd\mu(g_0)=\kappa^{-1}\tr(U(g))$.
\end{itemize}
\end{proposition}
\begin{proof}
By using Theorem \ref{thm-exprexp}, we can show that $\overline{\caE^\caO_g(g_0)}=\tr(U(g^{-1})\Omega_{\bfm_0}(g_0))=\caE^\caO_{g^{-1}}(g_0)$ since $\Omega_{\bfm_0}(g_0)$ is self-adjoint. In the same way, covariance follows the $G$-equivariance of $\Omega_{\bfm_0}$. BCH property is related to the fact that $U$ is a group representation. Finally, we get
\begin{equation*}
\int_\caO \caE^\caO_g(g_0)\dd\mu(g_0)=\tr(U(g)\int_\caO\Omega_{\bfm_0}(g_0)\dd\mu(g_0))=\kappa^{-1} \tr(U(g))
\end{equation*}
by using that $\Omega_{\bfm_0}(1)=\gone$.
\end{proof}
Note that the BCH property makes sense in a non-formal way only in the functional space $\caM_{\star_\theta}(G)$ determined in section \ref{subsec-mult}, where we will see that the star-exponential belongs to.

\subsection{Other determination by PDEs}
\label{subsec-other}

We give here another way to determine the star-exponential without using the quantization map, but directly by solving the PDE it has to satisfy. We restrict here to the case of an elementary normal $j$-group $G=\gS$ for simplicity.

By using the strong-invariance of the star-product, for any $f\in\caM_{\star_\theta}(\gS)$ (see section \ref{subsec-mult}),
\begin{equation*}
\forall X\in\ks\quad:\quad [\lambda_X,f]_{\star_\theta}=-i\theta X^\ast f,
\end{equation*}
where $\lambda$ is the moment map \eqref{eq-elem-moment}, and by using also the equivariance of $\Omega_{\bfm_0}$, we deduce that
\begin{multline*}
[\Omega_{\bfm_0}(\lambda_X),\Omega_{\bfm_0}(f)]=\Omega_{\bfm_0}([\lambda_X,f]_{\star_\theta})=-i\theta \Omega_{\bfm_0}(X^\ast f)=-i\theta\frac{d}{dt}|_0\Omega_{\bfm_0}(L^\ast_{e^{-tX}}f)\\
=-i\theta\frac{d}{dt}|_0 U(e^{tX})\Omega_{\bfm_0}(f)U(e^{-tX})=-i\theta[U_{\ast}(X),\Omega_{\bfm_0}(f)]
\end{multline*}
Since the center of $\caM_{\star_\theta}(\gS)$ is trivial, this means that there exists a linear map $\beta:\kg\to \gC$ such that
\begin{equation*}
\Omega_{\bfm_0}(\lambda_X)=-i\theta U_{\ast}(X)+\beta(X)\gone.
\end{equation*}
The invariance of the product under $\Sigma$ (see \eqref{eq-sigma}) implies that $\beta(X)=-\beta(X)$ and finally $\beta(X)=0$. As a consequence, we have the following proposition.
\begin{proposition}
The star-exponential (see Definition \ref{def-comput-starexp}) of an elementary normal $j$-group $G=\gS$ satisfies the equation:
\begin{equation}
\partial_t \caE_{e^{tX}}=\frac{i}{\theta}(\lambda_X\star_\theta \caE_{e^{tX}})\label{eq-other-eqdef}
\end{equation}
with initial condition $\lim_{t\to 0}\caE_{e^{tX}}=1$.
\end{proposition}
\begin{proof}
Indeed, by using $\Omega_{\bfm_0}(\lambda_X)=-i\theta U_{\ast}(X)$, we derive
\begin{multline*}
\partial_t \caE_{e^{tX}}(g_0)=\partial_t \tr(U(e^{tX})\Omega_{\bfm_0}(g_0))=\tr(U_{*}(X)U(e^{tX})\Omega_{\bfm_0}(g_0))\\
=\frac{i}{\theta}\tr(\Omega_{\bfm_0}(\lambda_X)U(e^{tX})\Omega_{\bfm_0}(g_0))=\frac{i}{\theta}(\lambda_X\star_\theta \caE_{e^{tX}})(g_0)
\end{multline*}
\end{proof}

Now, we can use this equation to find directly the expression of the star-exponential. Let us do it for example for the co-adjoint orbit associated to the sign $\ee=+1$. Since the equation \eqref{eq-other-eqdef} is integro-differential and complicated to solve, we will analyze the following equation
\begin{equation}
\partial_t f_t=\frac{i}{\theta}(\lambda_X\star_\theta^0 f_t),\qquad \lim_{t\to0} f_t=1\label{eq-expM1}
\end{equation}
for the Moyal product $\star_\theta^0$. Indeed, we have the expression of the intertwiner $T_\theta$ from $\star_\theta^0$ to $\star_\theta$. We define the partial Fourier transformation as
\begin{equation}
\caF f(a,x,\xi):=\hat f(a,x,\xi):=\int\dd\ell\ e^{-i\xi\ell}f(a,x,\ell).\label{eq-fourier}
\end{equation}
Applying the partial Fourier transformation \eqref{eq-fourier}, with $X=\alpha H+y+\beta E\in\ks$, on the action of moment maps by the Moyal product, we find
\begin{align*}
\caF(\lambda_{H}\star_\theta^0 f)&= \left(2i\partial_\xi+\frac{i\theta}{2}\partial_a\right)\hat f\\
\caF(\lambda_{y}\star_\theta^0 f)&=e^{-a-\frac{\theta\xi}{4}}\left(\omega_0(y,x)+\frac{i\theta}{2}y\partial_x\right)\hat f\\
\caF(\lambda_{E}\star_\theta^0 f)&=e^{-2a-\frac{\theta\xi}{2}}\hat f,
\end{align*}
so that the equation \eqref{eq-expM1} can be reformulated as
\begin{equation}
\partial_t \hat f_t=\frac{i}{\theta}\Big[2i\alpha\partial_\xi+\frac{i\theta\alpha}{2}\partial_a+\beta e^{-2a-\frac{\theta\xi}{2}}+ e^{-a-\frac{\theta\xi}{4}}(\omega_0(y,x)+\frac{i\theta}{2}y\partial_x)\Big]\hat f_t\label{eq-expM2}
\end{equation}
which is a pure PDE. Then, owing to the form of the moment map \eqref{eq-elem-moment}, we consider the ansatz
\begin{equation}
f_t(a,x,\ell)=v(t)\exp\frac{i}{\theta}\Big[2\ell\gamma_1(t)+e^{-2a}\gamma_2(t)+e^{-a}\gamma_3(t) \omega_0(y,x)\Big]\label{eq-expM3}
\end{equation}
whose partial Fourier transform can be expressed as
\begin{equation*}
\hat f_t(a,x,\xi)=4\pi^2\delta\Big(\xi-\frac{2\gamma_1(t)}{\theta}\Big)\,v(t)\exp\frac{i}{\theta}\Big[e^{-2a}\gamma_2(t)+e^{-a}\gamma_3(t)\omega_0(y,x)\Big].
\end{equation*}
Inserting this ansatz into Equation \eqref{eq-expM2}, it gives:
\begin{equation*}
\gamma'_1(t)=\alpha,\qquad \gamma'_2(t)=\alpha\gamma_2(t)+\beta e^{-\gamma_1(t)},\qquad
\gamma_3'(t)=\frac{\alpha}{2}\gamma_3(t)+e^{-\frac{\alpha t}{2}},\qquad v'(t)=0.
\end{equation*}
We find that the solutions with initial condition $\lim_{t\to0}f_t=1$ are:
\begin{equation*}
\gamma_1=\alpha t,\qquad \gamma_2=\frac{\beta}{\alpha}\sinh(\alpha t),\qquad \gamma_3=\frac{\sinh(\frac{\alpha t}{2})}{\alpha},\qquad v=1.
\end{equation*}

By using intertwining operators \eqref{eq-quelem-intertw}, we see that $T_\theta^{-1}\lambda_X=\lambda_X$, and $T_\theta f_t$ is then a solution of \eqref{eq-other-eqdef}:
\begin{eqnarray*}
&&E_{\star_\theta}(t\lambda_X)(a,x,\ell):=\caE_{e^{tX}}(a,x,\ell)=T_\theta f_t(a,x,\ell)\\
&=&\sqrt{\cosh(\alpha t)}\cosh(\frac{\alpha t}{2})^n e^{\frac{i}{\theta}\sinh(\alpha t)\Big(2\ell+\frac{\beta}{\alpha}e^{-2a}
+\frac{e^{-a}}{\alpha}\omega_0(y,x)\Big)}
\end{eqnarray*}
To obtain the star-exponential, we need the expression of the logarithm of the group $\gS$: $\caE_{g_0}=E_{\star_\theta}(\lambda_{\log(g_0)})$. For $X=\alpha H+y+\beta E\in\ks$, the exponential of the group $\gS$ has the expression
\begin{equation*}
\exp(\alpha H+y+\beta E)=\Big(\alpha,\frac{2e^{-\frac{\alpha}{2}}}{\alpha}\sinh(\frac{\alpha}{2})y, \frac{\beta}{\alpha}e^{-\alpha}\sinh(\alpha)\Big),
\end{equation*}
and the logarithm:
\begin{equation*}
\log(a,x,\ell)=a H+\frac{a}{2}\frac{e^{\frac{a}{2}}}{\sinh(\frac{a}{2})}x
+\frac{a e^a}{\sinh(a)}\ell E.
\end{equation*}
Therefore, we obtain
\begin{equation*}
\caE_{g_0}(g)= \sqrt{\cosh(a_0)}\cosh(\frac{a_0}{2})^n e^{\frac{i}{\theta}\Big(2\sinh(a_0)\ell+e^{a_0-2a}\ell_0
+e^{\frac{a_0}{2}-a}\cosh(\frac{a_0}{2})\omega_0(x_0,x)\Big)},
\end{equation*}
which coincides with the expression \eqref{eq-comput-starexpelem} determined by using the quantization map $\Omega_{\bfm_0}$. Note that the BCH property (see Proposition \ref{prop-link3}) can also be checked directly at the level of the Lie algebra $\ks$. From the above expressions of the logarithm and the exponential of the group $\gS$, we derive the BCH expression: $\text{BCH}(X_1,X_2):=\log(e^{X_1}e^{X_2})$, i.e.
\begin{multline*}
\text{BCH}(X_1,X_2)=\Big(\alpha_1+\alpha_2,\frac{(\alpha_1+\alpha_2)}{\sinh(\frac{\alpha_1+\alpha_2}{2})}\Big(\frac{e^{-\frac{\alpha_2}{2}}}{\alpha_1}\sinh(\frac{\alpha_1}{2})y_1+\frac{e^{\frac{\alpha_1}{2}}}{\alpha_2}\sinh(\frac{\alpha_2}{2})y_2\Big),\\
\frac{(\alpha_1+\alpha_2)}{\sinh(\alpha_1+\alpha_2)}\Big[\frac{\beta_1}{\alpha_1}e^{-\alpha_2}\sinh(\alpha_1)+\frac{\beta_2}{\alpha_2}e^{\alpha_1}\sinh(\alpha_2)+\frac{2}{\alpha_1\alpha_2}e^{\frac{\alpha_1-\alpha_2}{2}}\sinh(\frac{\alpha_1}{2})\sinh(\frac{\alpha_2}{2})\omega_0(y_1,y_2)\Big]\Big).
\end{multline*}
Then, BCH property $\caE_{g}\star_\theta\caE_{g'}=\caE_{g\fois g'}$ is equivalent to
\begin{equation*}
\forall X_1,X_2\in\ks\quad:\quad E_{\star_\theta}(\lambda_{X_1})\star_\theta E_{\star_\theta}(\lambda_{X_2})=E_{\star_\theta}(\lambda_{\text{BCH}(X_1,X_2)})
\end{equation*}
which turns out to be true for the star-product \eqref{eq-starprod} and the star-exponential determined above.

\section{Non-formal definition of the star-exponential}

\subsection{Schwartz spaces}
\label{subsec-schwartz}

In \cite{BG}, a Schwartz space adapted to the elementary normal $j$-group $\gS$ has been introduced, which is different from the usual one $\caS(\gR^{2n+2})$ in the global chart $\{(a,x,\ell)\}$, but related to oscillatory integrals. Let us have a look on the phase \eqref{eq-starprod} of the star-product:
\begin{equation*}
\ee S_{\gS}(0,g_1,g_2)=\sinh(2a_1)\ell_2-\sinh(2a_2)\ell_1+\cosh(a_1)\cosh(a_2)\omega_0(x_1,x_2)
\end{equation*}
with $g_i=(a_i,x_i,\ell_i)\in\gS$. Recall that the left-invariant vector fields of $\gS$ are given by
\begin{equation*}
\tilde H=\partial_a-x\partial_x-2\ell\partial_\ell,\quad \tilde y=y\partial_x+\frac12\omega_0(x,y)\partial_\ell,\quad \tilde E=\partial_\ell.
\end{equation*}
We define the maps $\tilde\alpha$ by $\forall X=(X_1,X_2)\in\ks\oplus\ks$,
\begin{equation*}
\tilde X\fois e^{-\frac{2i}{\theta}S_\gS(0,g_1,g_2)}=:-\frac{2i\ee}{\theta}\tilde\alpha_X(g_1,g_2)e^{-\frac{2i}{\theta}S_\gS(0,g_1,g_2)}
\end{equation*}
since it is an oscillatory phase. For example, we have
\begin{equation*}
\tilde\alpha_{(E,0)}(g_1,g_2)=-\sinh(2a_2),\; \tilde\alpha_{(H,0)}(g_1,g_2)=2\cosh(2a_1)\ell_2+2\sinh(2a_2)\ell_1-e^{-a_1}\cosh(a_2)\omega_0(x_1,x_2).
\end{equation*}
Then, we set $\alpha_X(g):=\tilde\alpha_{(X,0)}(0,g)$ for any $X\in\ks$ and $g\in G$, whose expressions are
\begin{equation*}
\alpha_H(g)=2\ell,\qquad \alpha_y(g)=\cosh(a)\omega(y,x),\qquad \alpha_E(g)=-\sinh(2a).
\end{equation*}
This leads to the following definition.
\begin{definition}
\label{def-schwartz}
The {\defin Schwartz space} of $\gS$ is defined as
\begin{equation*}
\caS(\gS)=\{f\in C^\infty(\gS)\quad \forall j\in\gN^{2n+2},\ \forall P\in\caU(\ks)\quad
\norm f\norm_{j,P}:=\sup_{g\in\gS}\Big| \alpha^j(g) \tilde P f(g)\Big|<\infty\}
\end{equation*}
where $\alpha^j:=\alpha_H^{j_1}\alpha_{e_1}^{j_2}\dots\alpha_{e_{2n}}^{j_{2n+1}} \alpha_E^{j_{2n+2}}$.
\end{definition}
\noindent It turns out that the space $\caS(\gS)$ corresponds to the usual Schwartz space in the coordinates $(r,x,\ell)$ with $r=\sinh(2a)$. It is stable by the action of $\gS$:
\begin{equation*}
\forall f\in\caS(\gS),\ \forall g\in \gS\quad:\quad g^\ast f\in\caS(\gS).
\end{equation*}
Moreover, $\caS(\gS)$ is a Fr\'echet nuclear space endowed with the seminorms $(\norm f\norm_{j,P})$.
\medskip

For $f,h\in\caS(\gS)$, the product $f\star_{\theta}h$ is well-defined by \eqref{eq-starprod}. However, to show that it belongs to $\caS(\gS)$, we will use arguments close to oscillatory integral theory. Let us illustrate this concept. One can show that the following operators leave the phase $e^{-\frac{2i}{\theta}S_\gS(0,g_1,g_2)}$ invariant:
\begin{align*}
&\caO_{a_2}:=\frac{1}{1+\tilde\alpha_{(E,0)}^2}(1-\frac{\theta^2}{4}\tilde E^2)=\frac{1}{1+\sinh(2a_2)^2}(1-\frac{\theta^2}{4}\partial_{\ell_1}^2), \qquad \caO_{a_1}:=\frac{1}{1+\sinh(2a_1)^2}(1-\frac{\theta^2}{4}\partial_{\ell_2}^2),\\
& \caO_{x_2}:=\frac{1}{1+x_2^2}(1-\frac{\theta^2}{4\cosh(a_1)^2\cosh(a_2)^2}\partial_{x_1}^2), \qquad  \caO_{x_1}:=\frac{1}{1+x_1^2}(1-\frac{\theta^2}{4\cosh(a_1)^2\cosh(a_2)^2}\partial_{x_2}^2),\\
&\caO_{\ell_2}:=\frac{1}{1+\ell_2^2}(1-\frac{\theta^2}{4}(\frac{1}{\cosh(2a_1)}(\partial_{a_1}-\tanh(a_1)x_1\partial_{x_1}))^2),\\
&\caO_{\ell_1}:=\frac{1}{1+\ell_1^2}(1-\frac{\theta^2}{4}(\frac{1}{\cosh(2a_2)}(\partial_{a_2}-\tanh(a_2)x_2\partial_{x_2}))^2).
\end{align*}

So we can add arbitrary powers of these operators in front of the phase without changing the expression. Then, using integrations by parts, we have for $F\in\caS(\gS^2)$:
\begin{multline}
\int\ e^{-\frac{2i}{\theta}S_\gS(0,g_1,g_2)} F(g_1,g_2)\dd g_1\dd g_2=\\
\int\ e^{-\frac{2i}{\theta}S_\gS(0,g_1,g_2)} (\caO^\ast_{a_1})^{k_1} (\caO^\ast_{a_2})^{k_2} (\caO^\ast_{x_1})^{p_1} (\caO^\ast_{x_2})^{p_2} (\caO^\ast_{\ell_1})^{q_1} (\caO^\ast_{\ell_2})^{q_2} F(g_1,g_2)\dd g_1\dd g_2\\
= \int\ e^{-\frac{2i}{\theta}S_\gS(0,g_1,g_2)} \frac{1}{(1+\sinh^2(2a_1))^{k_1}(1+\sinh^2(2a_2))^{k_2}}\\
\frac{1}{(1+x_1^2)^{p_1-q_2}(1+x_2^2)^{p_2-q_1}(1+\ell_1^2)^{q_1}(1+\ell_2^2)^{q_2}}DF(g_1,g_2)\dd g_1\dd g_2\label{eq-osc}
\end{multline}
for any $k_i,q_i,p_i\in\gN$ such that $p_1\geq q_2$ and $p_2\geq q_1$, and where $D$ is a linear combination of products of bounded functions (with every derivatives bounded) in $(g_1,g_2)$ with powers of $\partial_{\ell_i}$, $\partial_{x_i}$ and $\frac{1}{\cosh(2a_i)}\partial_{a_i}$. The first line of \eqref{eq-osc} is not defined for non-integrable functions $F$ bounded by polynoms in $r_i:=\sinh(2a_i)$, $x_i$ and $\ell_i$. However, the last two lines of \eqref{eq-osc} are well-defined for $k_i,p_i,q_i$ sufficiently large. Therefore it gives a sense to the first line, now understood as an oscillatory integral, i.e. as being equal to the last last two lines. This definition of oscillatory integral \cite{BM,BG} is unique, in particular unambiguous in the powers $k_i,p_i,q_i$ because of the density of $\caS(\gS)$ in polynomial functions in $(r,x,\ell)$ of a given degree. Note that this corresponds to the usual oscillatory integral \cite{Hormander:1979} in the coordinates $(r,x,\ell)$.

The next Theorem, proved in \cite{BG}, can be showed by using such methods of oscillatory integrals on $\caS(\gS)$.
\begin{theorem}
Let $\caP:\gR\to C^\infty(\gR)$ be a smooth map such that $\caP_0\equiv1$, and $\caP_\theta(a)$ as well as its inverse are bounded by $C\sinh(2a)^k$, $k\in\gN$, $C>0$. Then, the expression \eqref{eq-starprod} yields a $\gS$-invariant non-formal deformation quantization.\\
In particular, $(\caS(\gS),\star_{\theta})$ is a nuclear Fr\'echet algebra.
\end{theorem}

In the following, we show a factorization property for this Schwartz space. First, by introducing $\gamma(a)=\sinh(2a)$ and $\caS(A):=\gamma^\ast\caS(\gR)$, we note that the group law of $\gS$ reads in the coordinates $(r=\gamma(a),x,\ell)$:
\begin{multline*}
(r,x,\ell)\fois (r',x',\ell')=\Big(r\sqrt{1+r'^2}+r'\sqrt{1+r^2},(c(r')-s(r'))x+x',(\sqrt{1+r'^2}-r')\ell+\ell'\\
+\frac{1}{2}(c(r')-s(r'))\omega_0(x,x')\Big)
\end{multline*}
with the auxiliary functions:
\begin{align}
&c(r)=\frac{\sqrt 2}{2}(1+\sqrt{1+r^2})^{\frac12}=\cosh(\frac12\text{arcsinh}(r)),\label{eq-auxfunc}\\
&s(r)=\frac{\sqrt 2}{2}\text{sgn}(r)(-1+\sqrt{1+r^2})^{\frac12}=\sinh(\frac12\text{arcsinh}(r)).\nonumber
\end{align}

\begin{proposition}[Factorization]
\label{prop-factor}
The map $\Phi$ defined by $\Phi(f\otimes h)=f\star_\theta h$, for $f\in\caS(A)$ and $h\in\caS(\gR^{2n+1})$ realizes a continuous automorphism $\caS(\gS)=\caS(A)\hat\otimes\caS(\gR^{2n+1})\to\caS(\gS)$.
\end{proposition}
\begin{proof}
Due to the nuclearity of the Schwartz space, we have indeed $\caS(\gS)=\caS(A)\hat\otimes\caS(\gR^{2n+1})$. For $f\in\caS(A)$ (abuse of notation identifying $f(a)$ and $f(r):=f(\gamma^{-1}(r))$) and $h\in\caS(\gR^{2n+1})$, we reexpress the star-product \eqref{eq-starprod} in the coordinates $(r,x,\ell)$:
\begin{multline*}
(f\star_{\theta}h)(r,x,\ell)=\frac{1}{(\pi\theta)^{2n+2}}\int\ \Big(1-\frac{r_1r_2}{\sqrt{(1+r_1^2)(1+r_2^2)}}\Big)\frac{\sqrt{c(r_1)c(r_2)}}{\sqrt{c(r_1\sqrt{1+r_2^2}-r_2\sqrt{1+r_1^2})}}\\
e^{-\frac{2i\ee}{\theta}(r_1\ell_2-r_2\ell_1+\omega_0(x_1,x_2))} f(r\sqrt{1+r_1^2}+r_1\sqrt{1+r^2})\\
h(\frac{1}{c(r_1)}x_2+(c(r_2)-\frac{s(r_1)s(r_2)}{c(r_1)})x,\ell_2+\sqrt{1+r_2^2}\ell)\dd r_i\dd x_i\dd \ell_i
\end{multline*}
By using the partial Fourier transform $\hat h(r,\xi)=\int\dd\ell e^{-i\ell\xi}h(x,\ell)$, and integrating over several variables, we obtain
\begin{equation*}
(f\star_{\theta}h)(r,x,\ell)=\frac{1}{2\pi}\int\ f(r\sqrt{1+\frac{\theta^2\xi^2}{4}}+\frac{\ee\theta\xi}{2}\sqrt{1+r^2})\hat h(x,\xi)e^{i\ell\xi}\dd \xi.
\end{equation*}
For $\varphi\in\caS(\gS)$, we have now the following explicit expression for $\Phi$:
\begin{equation*}
\Phi(\varphi)(r,x,\ell)=\frac{1}{2\pi}\int\ \hat\varphi(r\sqrt{1+\frac{\theta^2\xi^2}{4}}+\frac{\ee\theta\xi}{2}\sqrt{1+r^2},x,\xi)e^{i\ell\xi}\dd \xi
\end{equation*}
which permits to deduce that $\Phi$ is valued in $\caS(\gS)$ and continuous. Then, the formula
\begin{equation*}
\hat \varphi(r,x,\ell)=\int\ \Phi(\varphi)(r\sqrt{1+\frac{\theta^2\ell^2}{4}}-\frac{\ee\theta\ell}{2}\sqrt{1+r^2},x,\xi)e^{-i\ell\xi}\dd\xi
\end{equation*}
permits to obtain the inverse of $\Phi$ which is also continuous.
\end{proof}

For normal $j$-groups $G=G_1\ltimes \gS_2$, we define the Schwartz space recursively
\begin{equation*}
\caS(G)=\caS(G_1)\hat\otimes\caS(\gS_2)
\end{equation*}
and obtain the same properties as before. In particular, endowed with the star-product \eqref{eq-starprodnorm}, the Schwartz space $\caS(G)$ is a nuclear Fr\'echet algebra.

\subsection{Multipliers}
\label{subsec-mult}

Let us consider the topological dual $\caS'(\gS)$ of $\caS(\gS)$. In the coordinates $(r=\gamma(a),x,\ell)$, it corresponds to tempered distributions. By denoting $\langle-,-\rangle$ the duality bracket between $\caS'(\gS)$ and $\caS(\gS)$, one can extend the product $\star_\theta$ (with tracial identity) as
\begin{equation*}
\forall T\in\caS'(\gS),\ \forall f,h\in\caS(\gS)\quad:\quad \langle T\star_\theta f,h\rangle:=\langle T,f\star_\theta h\rangle\text{ and } \langle f\star_\theta T,h\rangle:=\langle T,h\star_\theta f\rangle,
\end{equation*}
which is compatible with the case $T\in\caS(\gS)$.

\begin{definition}
The {\defin multiplier space} associated to $(\caS(\gS),\star_\theta)$ is defined as
\begin{equation*}
\caM_{\star_\theta}(\gS):=\{T\in\caS'(\gS),\ f\mapsto T\star_\theta f\text{ and } f\mapsto f\star_\theta T\text{ are continuous from }\caS(\gS) \text{ into itself}\}.
\end{equation*}
We can equipy this space with the topology associated to the seminorms:
\begin{equation*}
\norm T\norm_{B,j,P,L}=\sup_{f\in B}\norm T\star f\norm_{j,P}\,\text{ and }\, \norm T\norm_{B,j,P,R}=\sup_{f\in B}\norm f\star T\norm_{j,P}
\end{equation*}
where $B$ is a bounded subset of $\caS(\gS)$, $j\in\gN^{2n+2}$, $P\in\caU(\ks)$ and $\norm f\norm_{j,P}$ is the Schwartz seminorm introduced in Definition \ref{def-schwartz}. Note that $B$ can be described as a set satisfying $\forall j,P$, $\sup_{f\in B}\norm f\norm_{j,P}$ exists.
\end{definition}

\begin{proposition}
The star-product can be extended to $\caM_{\star_\theta}(\gS)$ by:
\begin{equation*}
\forall S,T\in\caM_{\star_\theta}(\gS),\ \forall f\in\caS(\gS)\quad:\quad \langle S\star_\theta T,f\rangle:=\langle S,T\star_\theta f\rangle=\langle T,f\star_\theta S\rangle.
\end{equation*}
Then, $(\caM_{\star_\theta}(\gS),\star_\theta)$ is an associative Hausdorff locally convex complete and nuclear algebra, with separately continuous product, called the multiplier algebra.
\end{proposition}
\begin{proof}
For the extension of the star-product and its associativity, we can show successively $\forall S,T\in\caM_{\star_\theta}(\gS)$, $\forall f,h\in\caS(\gS)$,
\begin{equation*}
(T\star_\theta f)\star_\theta h=T\star_\theta(f\star_\theta h),\quad (S\star_\theta T)\star_\theta f=S\star_\theta (T\star_\theta f),\quad (T_1\star_\theta T_2)\star_\theta T_3= T_1\star_\theta (T_2\star_\theta T_3),
\end{equation*}
each time by evaluating the distribution on a Schwartz function $\varphi\in\caS(\gS)$ and by using the factorization property (Proposition \ref{prop-factor}).

$\caM_{\star_\theta}(\gS)$ is the intersection of $\caM_L$, the left multipliers, and $\caM_R$, the right multipliers. By definition, each space $\caM_L$ and $\caM_R$ is topologically isomorphic to $\caL(\caS(\gS))$ endowed with the strong topology. Since $\caS(\gS)$ is Fr\'echet and nuclear, so is $\caL(\caS(\gS))$, as well as $\caM_L$, $\caM_R$ and finally $\caM_{\star_\theta}(\gS)$ (see \cite{Treves:1967} Propositions 50.1, 50.5 and 50.6).
\end{proof}

Due to the definition of $\caS(G)$ for a normal $j$-group $G=G_1\ltimes \gS_2$ and to the expression of the star-product \eqref{eq-starprodnorm}, the multiplier space associated to $(\caS(G),\star_\theta)$ takes the form:
\begin{equation}
\caM_{\star_\theta}(G)=\caM_{\star_\theta}(G_1)\hat\otimes\caM_{\star_\theta}(\gS_2)\label{eq-multnorm}
\end{equation}
and is also an associative Hausdorff locally convex complete and nuclear algebra, with separately continuous product. Remember that we have identified co-adjoint orbits $\caO$ described in Proposition \ref{prop-moment} with the group $G$ itself, so that we can speak also about the multiplier algebra $\caM_{\star_\theta}(\caO)$.

\subsection{Non-formal star-exponential}

\begin{theorem}
Let $G$ be a normal $j$-group and $\star_\theta$ the star-product \eqref{eq-starprodnorm}. Then for any $g\in G$, the star-exponential \eqref{eq-comput-starexp} $\caE^\caO_g$ lies in the multiplier algebra $\caM_{\star_\theta}(\caO)$.
\end{theorem}
\begin{proof}
Let us focus for the moment on the case of the elementary group $\caS$. The general case can then be obtained recursively due to the structure of the star-exponential \eqref{eq-comput-starexp} and of the multiplier algebra \eqref{eq-multnorm}. We use the same notations as before. For $f,h\in\caS(A)$, $f\star_\theta h=f\fois h$. If $T$ belongs to the multiplier space $\caM(\caS(A))$ of $\caS(A)$ for the usual commutative product, we have in particular $T\in\caS'(\gS)$ and by duality $T\star_\theta f=T\fois f$. Then,
\begin{equation*}
\forall f\in\caS(A),\ \forall h\in\caS(\gR^{2n+1})\quad:\quad T\star_\theta (f\star_\theta h)=(T\fois f)\star_\theta h.
\end{equation*}
By the factorization property (Proposition \ref{prop-factor}), it means that $T\in\caM_{\star_\theta}(\gS)$, and we have an embedding $\caM(\caS(A))\hookrightarrow\caM_{\star_\theta}(\gS)$. If we note as before $\gR^{2n+1}=V\oplus\gR E$, we can show in the same way that there is another embedding $\caM(\caS(\gR E))\hookrightarrow\caM_{\star_\theta}(\gS)$. Since $x'\in V\mapsto\caE^{\caO_2}_{(g_1,g_2)}(0,x',0)$ is an imaginary exponential of a polynom of degree less or equal than 2 in $x'$ and since the product $\star_\theta$ coincides with the Moyal product on $V$, it turns out that $x'\in V\mapsto\caE^{\caO_2}_{(g_1,g_2)}(0,x',0)$ is in $\caM_{\star_\theta}(\caS(V))$. Then, the star-exponential $\caE^{\caO_2}$ in \eqref{eq-comput-starexp} lies in $\caM(\caS(A))\hat\otimes \caM_{\star_\theta}(\caS(V))\hat\otimes\caM(\caS(\gR E))$, and it belongs also to $\caM_{\star_\theta}(\gS)$.
\end{proof}

\section{Adapted Fourier transformation}

\subsection{Definition}

As in the case of the Weyl-Moyal quantization treated in \cite{AC,Arnal:1988}, we can introduce the notion of adapted Fourier transformation. For normal $j$-groups $G=G_1\ltimes\gS_2$, which are not unimodular, it is relevant for that to introduce a {\defin modified star-exponential}
\begin{equation*}
\tilde\caE^{\caO}_g(g'):=\tr(U(g)d^{\frac12}\Omega(g')).
\end{equation*}
where $d$ is the formal dimension operator associated to $U$ (see \cite{Duflo:1976,Gayral:2008fo}) and $\caO$ is the co-adjoint orbit determining the irreducible representation $U$. Such an operator $d$ is used to regularize the expressions since $\int f(g)U(g)d^{\frac 12}$ is a Hilbert-Schmidt operator whenever $f$ is in $L^2(G)$. So the trace in the definition of $\tilde\caE^{\caO}$ is understood as a distribution only in the variable $g'\in \caO$.

By denoting $\Delta$ the modular function, defined by $\dd^L(g\fois g')=\Delta(g')\dd^L g$, whose computation gives
\begin{equation*}
\Delta(g)=\Delta_1(g_1)\Delta_2(g_2),\text{ with }\Delta_2(a_2,x_2,\ell_2)=e^{-2(n_2+1)a_2},
\end{equation*}
the operator $d$ is defined (up to a positive constant) by the relation:
\begin{equation*}
\forall g\in G\quad:\quad U(g)d U(g)^{-1}=\Delta(g)^{-1}d.
\end{equation*}
Since $\caR(g_1)d_2\caR(g_1)^{-1}=d_2$, it can therefore be expressed as $d=d_1\otimes d_2$, for $d_i$ the dimension operator associated to $U_i$, and with $\forall\varphi_2\in\ehH_2$, $\forall (a_0,v_0)\in Q_2$,
\begin{equation*}
(d_2\varphi_2)(a_0,v_0)=\kappa_{2}^2 e^{-2(n_2+1)a_0}\varphi_2(a_0,v_0)
\end{equation*}
where we recall that $\dim(\gS_2)=2(n_2+1)$. Note that $d_2$ is independent here of the choice of the irreducible representation $U_2$ ($\ee_2=\pm1$).

\begin{proposition}
\label{prop-adapt-mod}
The expression of the modified star-exponential can then be computed in the same notations as for Theorem \ref{thm-exprexp}:
\begin{multline*}
\tilde\caE^{\caO}_g(g')=\tilde\caE^{\caO_1}_{g_1}(g'_1) \frac{e^{(n_2+1)(\frac{a_2}{2}-a_2')}}{(\pi\theta)^{n_2+1}}\frac{\sqrt{\cosh(a_2)}\cosh(\frac{a_2}{2})^{n_2}|\det(\rho_+(g_1))|^{\frac12}}{|\det(1+\rho_+(g_1))|}\\
\exp(\frac{i\ee_2}{\theta}\Big[ 2\sinh(a_2)\ell'_2
+e^{a_2-2a_2'}\ell_2+X^TA_\rho X\Big]).
\end{multline*}
\end{proposition}

\begin{definition}
We can now define the {\defin adapted Fourier transformation}: for $f\in\caS(G)$ and $g'\in \caO$,
\begin{equation*}
\caF_\caO(f)(g'):=\int_G f(g)\tilde\caE^{\caO}_g(g')\dd^Lg.
\end{equation*}
\end{definition}
We see that this definition is a generalization of the usual (symplectic) Fourier transformation. For example in the case of the group $\gR^{2}$, the star-exponential associated to the Moyal product is indeed given by $\exp(\frac{2i}{\theta}(a\ell'-a'\ell))$.

\subsection{Fourier analysis}
\label{subsec-fourier}

\begin{proposition}
\label{prop-fourier-orth}
The modified star-exponential satisfies an orthogonality relation: for $g',g''\in G$,
\begin{equation*}
\int_G \overline{\tilde\caE^{\caO}_g(g')}\tilde\caE^{\caO}_g(g'')\dd^Lg= \frac{1}{\Delta(g'')}\delta(g''\fois (g')^{-1})
\end{equation*}
\end{proposition}
Note that $\Delta(g''_2)^{-1}\delta(g''_2\fois (g'_2)^{-1})=\delta(a''_2-a'_2)\delta(x''_2-x'_2)\delta(\ell''_2-\ell'_2)$. This orthogonality relation does not hold for the unmodified star-exponential.
\begin{proof}
We use the expression of Proposition \ref{prop-adapt-mod}:
\begin{multline*}
\int_G \overline{\tilde\caE^{\caO}_g(g')}\tilde\caE^{\caO}_g(g'')\dd^Lg= \int_{G_1} \overline{\tilde\caE^{\caO_1}_{g_1}(g'_1)}\tilde\caE^{\caO_1}_{g_1}(g''_1)\int_{\gS_2} \frac{e^{(n_2+1)(a_2-a_2'-a_2'')}}{(\pi\theta)^{2(n_2+1)}}\\
\frac{|\det(\rho_+(g_1))|\cosh(a_2)\cosh(\frac{a_2}{2})^{2n_2}}{|\det(1+\rho_+(g_1))|^2}e^{\frac{i\ee_2}{\theta}(2\sinh(a_2)(\ell''_2-\ell'_2)
+e^{a_2}(e^{-2a_2''}-e^{-2a_2'})\ell_2)}\\
e^{\frac{i\ee_2}{\theta}((X'')^TA_\rho X''-(X')^TA_\rho X')}\dd^Lg_2\dd^Lg_1
\end{multline*}
with
\begin{equation*}
X'=\begin{pmatrix} \frac{1}{\sqrt 2}e^{\frac{a_2}{2}-a'_2}v_2 \\ \frac{1}{\sqrt 2}e^{\frac{a_2}{2}-a'_2}w_2 \\ \sqrt 2 \cosh(\frac{a_2}{2})v'_2 \\ \sqrt 2 \cosh(\frac{a_2}{2})w'_2\end{pmatrix}\quad\text{ and }\quad X''=\begin{pmatrix} \frac{1}{\sqrt 2}e^{\frac{a_2}{2}-a''_2}v_2 \\ \frac{1}{\sqrt 2}e^{\frac{a_2}{2}-a''_2}w_2 \\ \sqrt 2 \cosh(\frac{a_2}{2})v''_2 \\ \sqrt 2 \cosh(\frac{a_2}{2})w''_2\end{pmatrix}.
\end{equation*}
Integration over $\ell_2$ leads to the contribution $\delta(a_2'-a_2'')$. Since $A_\rho$ depends only on $g_1$, and $a_2'=a_2''$, we see that the gaussian part in $(v_2,w_2)$ disappears and integration over these variables brings $\frac{|\det(1+\rho_+(g_1))|^2}{|\det(\rho_+(g_1))|}\delta(v_2'-v_2'')\delta(w_2'-w_2'')$. Eventually, integration on $a_2$ can be performed and we find
\begin{equation*}
\int_G \overline{\tilde\caE^{\caO}_g(g')}\tilde\caE^{\caO}_g(g'')\dd^Lg= \int_{G_1} \overline{\tilde\caE^{\caO_1}_{g_1}(g'_1)}\tilde\caE^{\caO_1}_{g_1}(g''_1)\dd^Lg_1 \, \Delta(g''_2)^{-1}\delta(g''_2\fois (g'_2)^{-1})
\end{equation*}
which leads to the result recursively.
\end{proof}

\begin{proposition}
The adapted Fourier transformation satisfies the following property: $\forall f_1,f_2\in\caS(G)$,
\begin{equation*}
\caF_\caO(f_1\times f_2)=\frac{\Delta^{\frac12}}{\kappa}\big(\Delta^{-\frac12}\caF_\caO(f_1)\big)\star_\theta \big(\Delta^{-\frac12}\caF_\caO(f_2)\big),
\end{equation*}
with $(f_1\times f_2)(g)=\int_G f_1(g')f_2((g')^{-1}g)\dd^Lg'$ the usual convolution.
\end{proposition}
\begin{proof}
Due to the BCH property (see Proposition \ref{prop-link3}) and to the computation of the modified star-exponential $\tilde\caE^{\caO}_g(g')=\tilde\caE^{\caO_1}_{g_1}(g'_1)\frac{\kappa_2}{\Delta_2(g_2(g_2')^{-2})^{\frac12}}\caE^{\caO_2}_g(g_2')$, we have the modified BCH property
\begin{equation*}
\tilde\caE^{\caO}_{g\fois g'}(g'')=\frac{\Delta(g'')^{\frac12}}{\kappa}\big(\Delta^{-\frac12}\tilde\caE^{\caO}_g\big)\star_\theta \big(\Delta^{-\frac12}\tilde\caE^{\caO}_{g'}\big)(g''),
\end{equation*}
which leads directly to the result by using the expression of the adapted Fourier transform and the convolution.
\end{proof}

As in Remark \ref{rmk-decomp}, we consider the co-adjoint orbit $\caO_{(\ee)}=\caO_{1,(\ee_1)}\times\caO_{2,\ee_2}$ of the normal $j$-group $G=G_1\ltimes\gS_2$ determined by the sign choices $(\ee)=((\ee_1),\ee_2)\in(\gZ_2)^N$, with $(\ee_1)\in(\gZ_2)^{N-1}$ and $\ee_2\in\gZ_2$. Due to Proposition \ref{prop-adapt-mod}, we can write the modified star-exponential as
\begin{equation*}
\tilde\caE^{\caO_{(\ee)}}_g(g')=\tilde\caE^{\caO_{1,(\ee_1)}}_{g_1}(g'_1)\,\tilde\caE^{\caO_{2,\ee_2}}_{(g_1,g_2)}(g'_2),
\end{equation*}
with $g=(g_1,g_2)\in G$ and $g'=(g_1',g_2')\in\caO_{(\ee)}$.

\begin{theorem}
\label{thm-four-inv}
We have the following inversion formula for the adapted Fourier transformation: for $f\in\caS(G)$ and $g\in G$,
\begin{equation*}
f(g)=\sum_{(\ee)\in(\gZ_2)^N}\int_{\caO_{(\ee)}}\overline{\tilde\caE^{\caO_{(\ee)}}_g(g')} \caF_{\caO_{(\ee)}}(f)(g')\dd\mu(g').
\end{equation*}
Moreover, the Parseval-Plancherel theorem is true:
\begin{equation*}
\int_G|f(g)|^2\dd^Lg=\sum_{(\ee)\in(\gZ_2)^N}\int_{\caO_{(\ee)}}|\caF_{\caO_{(\ee)}}(f)(g')|^2\dd\mu(g').
\end{equation*}
\end{theorem}
\begin{proof}
Let us show the dual property to Proposition \ref{prop-fourier-orth}, i.e.
\begin{equation}
\sum_{(\ee)\in(\gZ_2)^N}\int_{\caO_{(\ee)}}\overline{\tilde\caE^{\caO_{(\ee)}}_{g'}(g)} \tilde\caE^{\caO_{(\ee)}}_{g''}(g)\dd\mu(g)= \frac{1}{\Delta(g'')}\delta(g''\fois (g')^{-1}).\label{eq-fourier-orth2}
\end{equation}
First, we have
\begin{multline*}
\int_{\caO_{(\ee)}} \overline{\tilde\caE^{\caO_{(\ee)}}_{g'}(g)}\tilde\caE^{\caO_{(\ee)}}_{g''}(g)\dd\mu(g)= \int_{\caO_{1,(\ee_1)}} \overline{\tilde\caE^{\caO_{1,(\ee_1)}}_{g_1'}(g_1)}\tilde\caE^{\caO_{1,(\ee_1)}}_{g_1''}(g_1)\int_{\caO_{2,\ee_2}} \frac{e^{(n_2+1)(\frac{a_2'+a_2''}{2}-2a_2)}}{(\pi\theta)^{2(n_2+1)}}\\
\frac{|\det(\rho_+(g_1'))\det(\rho_+(g_1''))|^{\frac12}\sqrt{\cosh(a_2')\cosh(a_2'')}\cosh(\frac{a_2'}{2})^{n_2}\cosh(\frac{a_2''}{2})^{n_2}}{|\det(1+\rho_+(g_1'))\det(1+\rho_+(g_1''))|}\\
 e^{\frac{i\ee_2}{\theta}(-2\sinh(a_2')\ell_2+2\sinh(a_2'')\ell_2
-e^{a_2'-2a_2}\ell_2'+e^{a_2''-2a_2}\ell_2''+(X'')^TA_\rho(g_1'') X''-(X')^TA_\rho(g_1') X')}\dd\mu_2(g_2)\dd\mu_1(g_1)
\end{multline*}
with
\begin{equation*}
X'=\begin{pmatrix} \frac{1}{\sqrt 2}e^{\frac{a_2'}{2}-a_2}v_2' \\ \frac{1}{\sqrt 2}e^{\frac{a_2'}{2}-a_2}w'_2 \\ \sqrt 2 \cosh(\frac{a_2'}{2})v_2 \\ \sqrt 2 \cosh(\frac{a_2'}{2})w_2\end{pmatrix}\quad\text{ and }\quad X''=\begin{pmatrix} \frac{1}{\sqrt 2}e^{\frac{a_2''}{2}-a_2}v_2'' \\ \frac{1}{\sqrt 2}e^{\frac{a''_2}{2}-a_2}w_2'' \\ \sqrt 2 \cosh(\frac{a''_2}{2})v_2 \\ \sqrt 2 \cosh(\frac{a''_2}{2})w_2\end{pmatrix}.
\end{equation*}
We want to compute the sum over $(\ee)\in(\gZ_2)^N$ of such terms. By recurrence, we can suppose that
\begin{equation*}
\sum_{(\ee_1)\in(\gZ_2)^{N-1}}\int_{\caO_{1,(\ee_1)}} \overline{\tilde\caE^{\caO_{1,(\ee_1)}}_{g'_1}(g_1)}\tilde\caE^{\caO_{1,(\ee_1)}}_{g''_1}(g_1)\dd\mu_1(g_1)= \frac{1}{\Delta(g''_1)}\delta(g''_1\fois (g'_1)^{-1}),
\end{equation*}
which means that $g_1''=g_1'$ in the following. The integration over $\ell_2$ brings a contribution in $\delta(a_2'-a_2'')$. Since $g_1''=g_1'$ and $a_2'=a_2''$, the gaussian part in $(v_2,w_2)$ disappears and integration over these variables brings $\frac{|\det(1+\rho_+(g_1'))|^2}{|\det(\rho_+(g_1'))|}\delta(v_2'-v_2'')\delta(w_2'-w_2'')$. The remaining term is proportional to
\begin{equation*}
\sum_{\ee_2=\pm1}\int_\gR e^{a_2'-2a_2} e^{\frac{i\ee_2}{\theta}e^{a_2'-2a_2}(\ell_2''-\ell_2')}\dd a_2=\pi\theta\delta(\ell_2''-\ell_2').
\end{equation*}

The property \eqref{eq-fourier-orth2} permits to show the inversion formula
\begin{equation*}
\sum_{(\ee)\in(\gZ_2)^N}\int_{\caO_{(\ee)}}\overline{\tilde\caE^{\caO_{(\ee)}}_g(g')}\caF_{\caO_{(\ee)}} (f)(g')\dd\mu(g')=\sum_{(\ee)\in(\gZ_2)^N}\int \overline{\tilde\caE^{\caO_{(\ee)}}_g(g')} f(g'')\tilde\caE^{\caO_{(\ee)}}_{g''}(g')\dd^Lg''\dd\mu(g')=f(g),
\end{equation*}
as well as the Parseval-Plancherel theorem
\begin{multline*}
\sum_{(\ee)\in(\gZ_2)^N}\int_{\caO_{(\ee)}}|\caF_{\caO_{(\ee)}} (f)(g')|^2\dd\mu(g')=\sum_{(\ee)\in(\gZ_2)^N}\int \overline{f(g)\tilde\caE^{\caO_{(\ee)}}_g(g')}f(g'')\tilde\caE^{\caO_{(\ee)}}_{g''}(g')\dd^Lg''\dd^Lg\dd\mu(g')\\
=\int_G|f(g)|^2\dd^Lg.
\end{multline*}
\end{proof}

\begin{corollary}
The map
\begin{equation*}
\caF:=\bigoplus_{(\ee)\in(\gZ_2)^N}\caF_{\caO_{(\ee)}}:\, L^2(G,\dd^Lg)\to \bigoplus_{(\ee)} L^2(\caO_{(\ee)},\mu),
\end{equation*}
defined by $\caF(f):=\bigoplus_{(\ee)} (\caF_{\caO_{(\ee)}}f)$ realizes an isometric isomorphism.
\end{corollary}
\begin{proof}
From Proposition \ref{prop-fourier-orth}, we deduce that $\forall (\ee)\in (\gZ_2)^N$, $\caF_{\caO_{(\ee)}}\caF_{\caO_{(\ee)}}^*=\gone$. And the Parseval-Plancherel means that $\sum_{(\ee)}\caF_{\caO_{(\ee)}}^*\caF_{\caO_{(\ee)}}=\gone$. Moreover, we can show that $\forall (\ee),(\ee')\in(\gZ_2)^N$, with $(\ee)\neq(\ee')$, $\caF_{\caO_{(\ee)}}\caF_{\caO_{(\ee')}}^*=0$. Indeed, if $k\leq N$ is such that $\ee_k\neq\ee'_k$, then the computation of $\int_G\overline{\tilde\caE^{\caO_{(\ee')}}_g(g')}\tilde\caE^{\caO_{(\ee)}}_g(g'')\dd^Lg$ corresponds to have a factor $e^{\frac{i\ee_k}{\theta}e^{a_2}(e^{-2a_2''}+e^{-2a_2'})\ell_2}$ in the proof of Proposition \ref{prop-fourier-orth}. Integration over $\ell_2$ makes this expression vanishing.

For each $(\ee)\in(\gZ_2)^N$ (i.e. for each $(\ee)=(\ee_1,\dots,\ee_N)$ with $\ee_j=\pm1$), we will consider a function $f_{(\ee)}\in L^2(\caO_{(\ee)},\mu)$. We denote by $\bigoplus_{(\ee)} f_{(\ee)}$ the collection of these $2^N$ functions on the different orbits $\caO_{(\ee)}$. By using the three above properties and the fact that $\caF^*(\bigoplus_{(\ee)} f_{(\ee)})=\sum_{(\ee)}\caF_{\caO_{(\ee)}}^*(f_{(\ee)})$, we obtain that
\begin{equation*}
\caF^*\caF(f)=\sum_{(\ee)}\caF_{\caO_{(\ee)}}^*\caF_{\caO_{(\ee)}}(f)=f,\qquad \caF\caF^*(\bigoplus_{(\ee)}f_{(\ee)})=\bigoplus_{(\ee)}(\caF_{\caO_{(\ee)}}\caF_{\caO_{(\ee)}}^*f_{(\ee)})=\bigoplus_{(\ee)} f_{(\ee)}.
\end{equation*}
\end{proof}

\subsection{Fourier transformation and Schwartz spaces}

Given such an adapted Fourier transformation, we can wonder wether the Schwartz space $\caS(G)$ defined in \cite{BG} (see section \ref{subsec-schwartz}) is stable by this transformation, as it is true in the flat case: the usual transformation stabilizes the usual Schwartz space on $\gR^n$. However, the answer appears to be wrong here. Let us focus on the case of the elementary normal $j$-group $\gS$. The Schwartz space $\caS(\gS)$ of Definition \ref{def-schwartz} corresponds to the usual Schwartz space in the coordinates $(r=\sinh(2a),x,\ell)$. These coordinates are adapted to the phase of the kernel of the star-product \eqref{eq-starprod}. For the star-exponential of $\gS$ given in \eqref{eq-comput-starexpelem}, we need also to consider the coordinates corresponding to the moment maps \eqref{eq-elem-moment}:
\begin{equation*}
\mu:\gS\to\gR_+^\ast\times\gR^{2n+1},\quad (a,x,\ell)\mapsto (e^{-2a},e^{-a}x,\ell).
\end{equation*}
We will denote the new variables $(s,z,\ell)=\mu(a,x,\ell)$.
\begin{definition}
We define the {\defin moment-Schwartz space} of $\gS$ to be
\begin{equation*}
\caS_\lambda(\gS)=\{f\in C^\infty(\gS)\quad (\mu^{-1})^\ast f\in\caS(\gR_+^\ast\times\gR^{2n+1})\text{ and } s^{-\frac{n+1}{2}}(\mu^{-1})^\ast f(s,z,\ell)\text{ is smooth in } s=0\}.
\end{equation*}
The space $\caS_\lambda(\gS)$ corresponds to the usual Schwartz space in the coordinates $(s,z,\ell)$ (for $s>0$) with some boundary regularity condition in $s=0$. As before, we identify the group $\gS$ with the co-adjoint orbit $\caO_\ee$ ($\ee=\pm1$).
\end{definition}

\begin{theorem}
The adapted Fourier transformation restricted to the Schwartz space induces an isomorphism
\begin{equation*}
\caF:\caS(\gS)\to\caS_\lambda(\caO_+)\oplus\caS_\lambda(\caO_-).
\end{equation*}
\end{theorem}
\begin{proof}
Let $f\in\caS(\gS)$. The Fourier transform reads:
\begin{multline*}
\caF_{\caO_\ee}(f)(s,z,\ell)=\frac{1}{(\pi\theta)^{n+1}}\int \dd r'\dd x'\dd \ell'\frac{f(r',x',\ell')}{(1+r'^2)^{\frac14}} (\sqrt{1+r'^2}+r')^{\frac{n+1}{2}} s^{\frac{n+1}{2}} c(r')^n\\
 e^{\frac{i\ee}{\theta}\Big(2r'\ell+(\sqrt{1+r'^2}+r')s\ell'+\frac12(\sqrt{1+r'^2}+r'+1)\omega_0(x',z)\Big)}
\end{multline*}
Here we use the function $c(r')$ defined in \eqref{eq-auxfunc}, the coordinates $s=e^{-2a}$, $z=e^{-a}x$, $r'=\sinh(a')$ and the fact that $f$ is Schwartz in the variable $\sinh(a)$ if and only if it is in the variable $\sinh(2a)$. We denote again by $f$ the function in the new coordinates by a slight abuse of language. We have to check that $h(s,z,\ell)=s^{-\frac{n+1}{2}}\caF_{\caO_\ee}(f)(s,z,\ell)$ is Schwartz in $(s,z,\ell)$, i.e. we want to estimate expressions of the type
\begin{equation*}
\int \dd s\dd z\dd\ell\ |(1+s^2)^{k_1} (1+z^2)^{p_1}(1+\ell^2)^{q_1}\partial_s^{k_2}\partial_{z}^{p_2}\partial_{\ell}^{q_2}\ h(s,z,\ell)|.
\end{equation*}
Let us provide an analysis in terms of oscillatory integrals.
\begin{itemize}
\item \underline{Polynom in $\ell$}: controled by an adapted power of the following operator (invariant acting on the phase) $\frac{1}{1+\ell^2}(1-\frac{\theta^2}{4}\Big(\partial_{r'}-\frac{\ell'}{\sqrt{1+r'^2}}\partial_{\ell'}+\frac{(\sqrt{1+r'^2}+r')}{\sqrt{1+r'^2}(\sqrt{1+r'^2}+r'+1)}x'\partial_{x'}\Big)^2)$ (see section \ref{subsec-schwartz}). Indeed, powers and derivatives in the variables $r',x',\ell'$ are controled by the Schwartz function $f$ inside the integral.
\item \underline{Polynom in $z$}: controled by an adapted power of the (invariant) operator $\frac{1}{1+z^2}(1-\frac{4\theta^2}{(\sqrt{1+r'^2}+r'+1)^2}\partial_{x'}^2)$. 
\item \underline{Polynom in $s$}: controled by an adapted power of the (invariant) operator $\frac{1}{1+s^2}(1-\frac{\theta^2}{(\sqrt{1+r'^2}+r')^2}\partial_{\ell'}^2)$. Note that the function $\frac{1}{(\sqrt{1+r'^2}+r')^2}$ is estimated by a polynom in $r'$ for $r'\to\pm\infty$, as its derivatives.
\item \underline{Derivations in $s$}: produce terms like powers of $(\sqrt{1+r'^2}+r')\ell'$ which are controled.
\item \underline{Derivations in $z$}: produce terms like powers of $(\sqrt{1+r'^2}+r'+1)x'$ which are controled.
\item \underline{Derivations in $\ell$}: produce terms like powers of $r'$.
\end{itemize}
This shows that $h$ is Schwartz in $(s,z,\ell)$, so $\caF(f)\in\caS_\lambda(\gS)$.

Conversely, let $f_\ee\in\caS_\lambda(\caO_\ee)$. Due to Theorem \ref{thm-four-inv}, we can write the inverse of the Fourier transform as:
\begin{multline*}
\caF^{-1}(f_+,f_-)(r,x,\ell)=\sum_{\ee=\pm1}\frac{1}{2(\pi\theta)^{n+1}}\int \dd s'\dd z'\dd \ell'\frac{f_\ee(s',z',\ell')}{s'^{\frac{n+1}{2}}} (\sqrt{1+r^2}+r)^{\frac{n+1}{2}}\sqrt{1+r^2} c(r)^n\\
 e^{-\frac{i\ee}{\theta}\Big(2r\ell'+(\sqrt{1+r^2}+r)s'\ell+\frac12(\sqrt{1+r^2}+r+1)\omega_0(x,z')\Big)}
\end{multline*}
Here we use now the coordinates $s'=e^{-2a'}$, $z'=e^{-a'}x'$, $r=\sinh(a)$. We want to estimate expressions of the type
\begin{equation*}
\int \dd r\dd x\dd\ell\ |(1+r^2)^{k_1} (1+x^2)^{p_1}(1+\ell^2)^{q_1}\partial_r^{k_2}\partial_{x}^{p_2}\partial_{\ell}^{q_2}\ \caF^{-1}(f_+,f_-)(r,x,\ell)|.
\end{equation*}
Let us provide also an analysis in terms of oscillatory integrals.
\begin{itemize}
\item \underline{Polynom in $\ell$}: controled by an adapted power of the following operator (invariant acting on the phase) $\frac{1}{1+\ell^2}(1-\frac{\theta^2}{(\sqrt{1+r^2}+r)^2}\partial_{s'}^2)$. As before, powers and derivatives in the variables $s',z',\ell'$ are controled by the Schwartz function $f$ inside the integral. Note that $\frac{f_\ee(s',z',\ell')}{s'^{\frac{n+1}{2}}}$ is smooth in $s=0$ so that the integral is well-defined for $s\in\gR_+$.
\item \underline{Polynom in $x$}: controled by an adapted power of the (invariant) operator: $$\frac{1}{1+x^2}(1-\frac{4\theta^2}{(\sqrt{1+r^2}+r+1)^2}\partial_{z'}^2)\;.$$ 
\item \underline{Polynom in $r$}: controled by an adapted power of the (invariant) operator $\frac{1}{1+r^2}(1-\frac{\theta^2}{4}\partial_{\ell'}^2)$.
\item \underline{Derivations in $r$}: produce terms like powers of $(\sqrt{1+r^2}+r),\frac{1}{\sqrt{1+r^2}},$ $r,c'(r),\ell',\frac{(\sqrt{1+r^2}+r')}{\sqrt{1+r^2}}s'\ell,$ $\omega_0(x,z'),...$ which are controled (see just above).
\item \underline{Derivations in $x$}: produce terms like powers of $(\sqrt{1+r^2}+r+1)z'$ which are controled.
\item \underline{Derivations in $\ell$}: produce terms like powers of $(\sqrt{1+r^2}+r)s'$ which are also controled.
\end{itemize}
This shows that $\caF^{-1}(f_+,f_-)\in\caS(\gS)$.
\end{proof}

\subsection{Application to noncommutative Baumslag-Solitar tori}

We consider the decomposition of $G$ into elementary normal $j$-groups of section \ref{subsec-normal}
\begin{equation*}
G=\big(\dots (\gS_1\ltimes_{\rho_1}\gS_2)\ltimes_{\rho_2}\dots\big)\ltimes_{\rho_{N-1}}\gS_N
\end{equation*}
and the associated basis
\begin{equation*}
\kB:=\Big(H_1,(f_1^{(i)})_{1\leq i\leq 2n_1},E_1,\dots,H_N,(f_N^{(i)})_{1\leq i\leq 2n_N},E_N\Big)
\end{equation*}
of its Lie algebra $\kg$, where $(f_j^{(i)})_{1\leq i\leq 2n_j}$ is a canonical basis of the symplectic space $V_j$ contained in $\gS_j$. We note $G_{\text{BS}}$ the subgroup of $G$ generated by $\{e^{\theta X},\ X\in\kB\}$ and call it the {\defin Baumslag-Solitar subgroup} of $G$. Indeed, in the case of the ``$ax+b$'' group (two-dimensional elementary normal $j$-group), and if $e^{2\theta}\in\gN$, this subgroup corresponds to the Baumslag-Solitar group \cite{Baumslag:1962}:
\begin{equation*}
\text{BS}(1,m):=\langle\; e_1,e_2\quad |\quad e_1e_2(e_1)^{-1}=(e_2)^m\;\rangle.
\end{equation*}

We have seen before that the star-exponential associated to a co-adjoint orbit $\caO$ is a group morphism $\caE:G\to\caM_{\star_\theta}(\caO)\simeq\caM_{\star_\theta}(G)$. Composed with the quantization map $\Omega$, it coincides with the unitary representation $U=\Omega\circ\caE$. So, if we take now the subalgebra of $\caM_{\star_\theta}(G)$ generated by the star-exponential of $G_{\text{BS}}$, i.e. by elements $\{\caE_{e^{\theta X}},\ X\in\kB\}$, then it is closed for the complex conjugation and it can be completed into a C*-algebra $\algA_G$ with norm $\norm \Omega(\fois)\norm_{\caL(\ehH)}$. This C*-algebra is canonically associated to the group $G$. Moreover, if $\theta\to 0$, this C*-algebra is commutative and corresponds thus to a certain torus.
\begin{definition}
Let $G$ be a normal $j$-group. We define the {\defin noncommutative Baumslag-Solitar torus} of $G$ to be the C*-algebra $\algA_G$ constructed above.
\end{definition}
It turns out that the relation between the generators $\caE_{e^{\theta X}}$ ($\ X\in\kB$) of $\algA_G$ can be computed explicitely by using the BCH formula of Proposition \ref{prop-link3}. Let us see some examples.

\begin{example}
In the elementary group case $G=\gS$, let
\begin{align*}
U(a,x,\ell)&:=\caE_{(\theta,0,0)}(a,x,\ell)=\sqrt{\cosh(\theta)}\cosh(\frac{\theta}{2})^n e^{\frac{2i}{\theta}\sinh(\theta)\ell},\\
V(a,x,\ell)&:=\caE_{(0,0,\theta)}(a,x,\ell)=e^{ie^{-2a}},\\
W_i(a,x,\ell)&:=\caE_{(0,\theta e_i,0)}(a,x,\ell)=e^{ie^{-a}\omega_0(e_i,x)},
\end{align*}
where $(e_i)$ is a canonical basis of the symplectic space $(V,\omega_0)$ of dimension $2n$ (i.e. $\omega_0(e_i,e_{i+n})=1$ if $i\leq n$). Then, we can compute relations like
\begin{equation*}
U\star_\theta V=V^{e^{2\theta}}\star_\theta U.
\end{equation*}
by using BCH property of the star-exponential (see Proposition \ref{prop-link3}). We obtain (by omitting the notation $\star$):
\begin{align*}
&UV=V^{e^{2\theta}} U\qquad (\text{ and }UV^{\beta}=V^{\beta e^{2\theta}}U),\\
&UW_i=W_i^{e^{\theta}}U,\qquad W_iW_{i+n}=V^{\theta}W_{i+n}W_i
\end{align*}
where the other commutation relations are trivial. Note that these relations become trivial at the commutative limit $\theta\to0$. In the two-dimensional case, where $\gS$ is the ``$a$x$+b$ group'', the relation $UV=V^{e^{2\theta}} U$ has already been obtained in another way in \cite{Iochum:2011}.
\end{example}

\begin{example}
Let us consider the Siegel domain of dimension $6$ (see Example \ref{ex-normal-siegel} for definitions and notations). As before, we can define the following generators:
\begin{align*}
U(g)&:=\caE_{(0,0,\theta,0,0,0)}(g)=\sqrt{\cosh(\theta)}\cosh(\frac{\theta}{2}) e^{\frac{2i}{\theta}\sinh(\theta)\ell_2},\\
V(g)&:=\caE_{(0,0,0,0,0,\theta)}(g)=e^{ie^{-2a_2}},\\
W_1(g)&:=\caE_{(0,0,0,\theta,0,0)}(g)=e^{ie^{-a_2}w_2},\\
W_2(g)&:=\caE_{(0,0,0,0,\theta,0)}(g)=e^{-ie^{-a_2}v_2},\\
R(g)&:=\caE_{(\theta,0,0,0,0,0)}(g)=\frac{e^{\frac{\theta}{2}}\sqrt{\cosh(\theta)}}{\cosh(\frac{\theta}{2})}e^{\frac{2i}{\theta}(\sinh(\theta)\ell_1+\tanh(\frac{\theta}{2})v_2w_2)},\\
S(g)&:=\caE_{(0,\theta,0,0,0,0)}(g)=e^{i(e^{-2a_1}+\frac12v_2^2)}.
\end{align*}
We obtain the relationship:
\begin{align*}
&UV=V^{e^{2\theta}} U,\qquad UW_1=(W_1)^{e^{\theta}}U,\qquad UW_2=(W_2)^{e^{\theta}}U,\qquad W_1W_2=V^{\theta}W_2W_1,\\
& RS=S^{e^{2\theta}}R,\qquad RW_1=(W_1)^{e^{\theta}}R,\qquad RW_2=(W_1)^{e^{-\theta}}R,\qquad SW_1=V^{\frac{\theta^2}{2}}(W_2)^\theta W_1S,
\end{align*}
where the other commutation relations are trivial.
\end{example}


\begin{thebibliography}{MM65}

\bibitem{Arnal:1988}
D.~Arnal, ``{The *-exponential},'' {\em in Quantum Theories and Geometry, Math.
  Phys. Studies} {\bf 10} (1988)  23--52.


\bibitem{AC} D.~Arnal and J.-C.~Cortet, ``{Repr\'esentations $\star$ des groupes exponentiels},'' {\em J. Funct. Anal.} {\bf 92} (1990) 103--175.

  \bibitem{Baumslag:1962}
G.~Baumslag and D.~Solitar, ``{Some two-generator one-relator non-Hopfian groups},'' {\em Bull. Amer. Math. Soc.} {\bf 68} (1962)  199--201.

\bibitem{BUCB} P.~Bieliavsky, \emph{Symmetric spaces and star-representations}, talk given at the \emph{West Coast Lie Theory Seminar}, U.C. Berkeley, November 1995.

\bibitem{Bieliavsky:2002}
P.~Bieliavsky, ``{Strict Quantization of Solvable Symmetric Spaces},'' {\em J.
  Sympl. Geom.} {\bf 1} (2002)  269--320.
  
\bibitem{B07} P.~Bieliavsky, ``{Non-formal deformation quantizations of solvable Ricci-type symplectic symmetric spaces},'' Proceedings of the workshop `Non-Commutative Geometry
and Physics' (Orsay 2007), {\em J. Phys.: Conf. Ser.} {\bf 103} (2008).


\bibitem{BG}
P.~Bieliavsky and V.~Gayral, ``{Deformation Quantization for Actions of
  Kahlerian Lie Groups},'' {\em to appear in Mem.
  Amer. Math. Soc.} (2013)  ,
\href{http://arxiv.org/abs/1109.3419}{{\tt arXiv:1109.3419 [math.OA]}}.

  \bibitem{Bieliavsky:2010su}
P.~Bieliavsky, A.~de~Goursac and G.~Tuynman, ``{Deformation quantization for Heisenberg supergroup},'' {\em J. Funct. Anal.} {\bf 263} (2012)  549--603.

\bibitem{BM}
P.~Bieliavsky and M.~Massar,``{Oscillatory integral formulae for left-invariant star products 
	on a class of Lie groups},''
 {\em Lett. Math. Phys.} {\bf 58} (2001)  115--128. 
  
  \bibitem{Duflo:1976}
M.~Duflo and C.~C. Moore, ``{On the regular representation of a nonunimodular
  locally compact group},'' {\em J. Funct. Anal.} {\bf 21} (1976)  209--243.

  
 \bibitem{F} C.~Fronsdal, ``{Some ideas about quantization},''
{\em Rep. Math. Phys.} {\bf 15} (1979) 111--175. 
    
  \bibitem{Gayral:2008fo}
V.~Gayral, J.~M. Gracia-Bondia and J.~C. Varilly, ``{Fourier analysis on the
  affine group, quantization and noncompact Connes geometries},'' {\em J.
  Noncommut. Geom.} {\bf 2} (2008)  215--261.


\bibitem{Hormander:1979}
L.~Hormander, ``{The Weyl calculus of pseudo-differential operators},'' {\em
  Comm. Pure Appl. Math.} {\bf 32} (1979)  359.
 

\bibitem{Iochum:2011}
B.~Iochum, T.~Masson, T.~Schucker and A.~Sitarz, ``{Compact $\theta$-deformation
  and spectral triples},'' {\em Rep. Math. Phys.} {\bf 68} (2011) 37--64.

 \bibitem{L} O. Loos,   {\it Symmetric spaces I: General Theory}. Benjamin, 1969.

  \bibitem{PS}
I.~I. Pyatetskii-Shapiro, {\em {Automorphic functions and the geometry of
  classical domains}}.
\newblock Mathematics and Its Applications, Vol. 8, Gordon and Breach Science
  Publishers, New York, 1969.
  
  \bibitem{Rieffel:1993}
M.~A. Rieffel, ``{Deformation Quantization for actions of R(D)},'' {\em Mem. Amer. Math. Soc.} {\bf 106} (1993)  R6.

\bibitem{Treves:1967}
F.~Treves, {\em {Topological vector spaces, distributions and kernels}}.
\newblock Academic Press, 1967.


\end{thebibliography}
\end{document}